\documentclass[11pt]{article}

\usepackage{amsthm} 

\newtheorem{theorem}{Theorem}[section]
\newtheorem{lemma}[theorem]{Lemma}
\newtheorem{corollary}[theorem]{Corollary}

\makeindex \makeglossary
\begin{document}
%
%

\long\def\ig#1{\relax}
\ig{Thanks to Roberto Minio for this def'n.  Compare the def'n of
\comment in AMSTeX.}

\newcount \coefa
\newcount \coefb
\newcount \coefc
\newcount\tempcounta
\newcount\tempcountb
\newcount\tempcountc
\newcount\tempcountd
\newcount\xext
\newcount\yext
\newcount\xoff
\newcount\yoff
\newcount\gap%
\newcount\arrowtypea
\newcount\arrowtypeb
\newcount\arrowtypec
\newcount\arrowtyped
\newcount\arrowtypee
\newcount\height
\newcount\width
\newcount\xpos
\newcount\ypos
\newcount\run
\newcount\rise
\newcount\arrowlength
\newcount\halflength
\newcount\arrowtype
\newdimen\tempdimen
\newdimen\xlen
\newdimen\ylen
\newsavebox{\tempboxa}%
\newsavebox{\tempboxb}%
\newsavebox{\tempboxc}%

\makeatletter
\setlength{\unitlength}{.01em}%
\def\settypes(#1,#2,#3){\arrowtypea#1 \arrowtypeb#2 \arrowtypec#3}
\def\settoheight#1#2{\setbox\@tempboxa\hbox{#2}#1\ht\@tempboxa\relax}%
\def\settodepth#1#2{\setbox\@tempboxa\hbox{#2}#1\dp\@tempboxa\relax}%
\def\settokens[#1`#2`#3`#4]{%
     \def\tokena{#1}\def\tokenb{#2}\def\tokenc{#3}\def\tokend{#4}}
\def\setsqparms[#1`#2`#3`#4;#5`#6]{%
\arrowtypea #1
\arrowtypeb #2
\arrowtypec #3
\arrowtyped #4
\width #5
\height #6
}
\def\setpos(#1,#2){\xpos=#1 \ypos#2}

\def\bfig{\begin{picture}(\xext,\yext)(\xoff,\yoff)}
\def\efig{\end{picture}}

\def\putbox(#1,#2)#3{\put(#1,#2){\makebox(0,0){$#3$}}}

\def\settriparms[#1`#2`#3;#4]{\settripairparms[#1`#2`#3`1`1;#4]}%

\def\settripairparms[#1`#2`#3`#4`#5;#6]{%
\arrowtypea #1
\arrowtypeb #2
\arrowtypec #3
\arrowtyped #4
\arrowtypee #5
\width #6
\height #6
}

\def\resetparms{\settripairparms[1`1`1`1`1;500]\width 500}

\resetparms

\def\mvector(#1,#2)#3{
\put(0,0){\vector(#1,#2){#3}}%
\put(0,0){\vector(#1,#2){30}}%
}
\def\evector(#1,#2)#3{{
\arrowlength #3
\put(0,0){\vector(#1,#2){\arrowlength}}%
\advance \arrowlength by-30
\put(0,0){\vector(#1,#2){\arrowlength}}%
}}

\def\horsize#1#2{%
\settowidth{\tempdimen}{$#2$}%
#1=\tempdimen
\divide #1 by\unitlength
}

\def\vertsize#1#2{%
\settoheight{\tempdimen}{$#2$}%
#1=\tempdimen
\settodepth{\tempdimen}{$#2$}%
\advance #1 by\tempdimen
\divide #1 by\unitlength
}

\def\vertadjust[#1`#2`#3]{%
\vertsize{\tempcounta}{#1}%
\vertsize{\tempcountb}{#2}%
\ifnum \tempcounta<\tempcountb \tempcounta=\tempcountb \fi
\divide\tempcounta by2
\vertsize{\tempcountb}{#3}%
\ifnum \tempcountb>0 \advance \tempcountb by20 \fi
\ifnum \tempcounta<\tempcountb \tempcounta=\tempcountb \fi
}

\def\horadjust[#1`#2`#3]{%
\horsize{\tempcounta}{#1}%
\horsize{\tempcountb}{#2}%
\ifnum \tempcounta<\tempcountb \tempcounta=\tempcountb \fi
\divide\tempcounta by20
\horsize{\tempcountb}{#3}%
\ifnum \tempcountb>0 \advance \tempcountb by60 \fi
\ifnum \tempcounta<\tempcountb \tempcounta=\tempcountb \fi
}

\ig{ In this procedure, #1 is the paramater that sticks out all the way,
#2 sticks out the least and #3 is a label sticking out half way.  #4 is
the amount of the offset.}

\def\sladjust[#1`#2`#3]#4{%
\tempcountc=#4
\horsize{\tempcounta}{#1}%
\divide \tempcounta by2
\horsize{\tempcountb}{#2}%
\divide \tempcountb by2
\advance \tempcountb by-\tempcountc
\ifnum \tempcounta<\tempcountb \tempcounta=\tempcountb\fi
\divide \tempcountc by2
\horsize{\tempcountb}{#3}%
\advance \tempcountb by-\tempcountc
\ifnum \tempcountb>0 \advance \tempcountb by80\fi
\ifnum \tempcounta<\tempcountb \tempcounta=\tempcountb\fi
\advance\tempcounta by20
}

\def\putvector(#1,#2)(#3,#4)#5#6{{%
\xpos=#1
\ypos=#2
\run=#3
\rise=#4
\arrowlength=#5
\arrowtype=#6
\ifnum \arrowtype<0
    \ifnum \run=0
        \advance \ypos by-\arrowlength
    \else
        \tempcounta \arrowlength
        \multiply \tempcounta by\rise
        \divide \tempcounta by\run
        \ifnum\run>0
            \advance \xpos by\arrowlength
            \advance \ypos by\tempcounta
        \else
            \advance \xpos by-\arrowlength
            \advance \ypos by-\tempcounta
        \fi
    \fi
    \multiply \arrowtype by-1
    \multiply \rise by-1
    \multiply \run by-1
\fi
\ifnum \arrowtype=1
    \put(\xpos,\ypos){\vector(\run,\rise){\arrowlength}}%
\else\ifnum \arrowtype=2
    \put(\xpos,\ypos){\mvector(\run,\rise)\arrowlength}%
\else\ifnum\arrowtype=3
    \put(\xpos,\ypos){\evector(\run,\rise){\arrowlength}}%
\fi\fi\fi
}}

\def\putsplitvector(#1,#2)#3#4{
\xpos #1
\ypos #2
\arrowtype #4
\halflength #3
\arrowlength #3
\gap 140
\advance \halflength by-\gap
\divide \halflength by2
\ifnum \arrowtype=1
    \put(\xpos,\ypos){\line(0,-1){\halflength}}%
    \advance\ypos by-\halflength
    \advance\ypos by-\gap
    \put(\xpos,\ypos){\vector(0,-1){\halflength}}%
\else\ifnum \arrowtype=2
    \put(\xpos,\ypos){\line(0,-1)\halflength}%
    \put(\xpos,\ypos){\vector(0,-1)3}%
    \advance\ypos by-\halflength
    \advance\ypos by-\gap
    \put(\xpos,\ypos){\vector(0,-1){\halflength}}%
\else\ifnum\arrowtype=3
    \put(\xpos,\ypos){\line(0,-1)\halflength}%
    \advance\ypos by-\halflength
    \advance\ypos by-\gap
    \put(\xpos,\ypos){\evector(0,-1){\halflength}}%
\else\ifnum \arrowtype=-1
    \advance \ypos by-\arrowlength
    \put(\xpos,\ypos){\line(0,1){\halflength}}%
    \advance\ypos by\halflength
    \advance\ypos by\gap
    \put(\xpos,\ypos){\vector(0,1){\halflength}}%
\else\ifnum \arrowtype=-2
    \advance \ypos by-\arrowlength
    \put(\xpos,\ypos){\line(0,1)\halflength}%
    \put(\xpos,\ypos){\vector(0,1)3}%
    \advance\ypos by\halflength
    \advance\ypos by\gap
    \put(\xpos,\ypos){\vector(0,1){\halflength}}%
\else\ifnum\arrowtype=-3
    \advance \ypos by-\arrowlength
    \put(\xpos,\ypos){\line(0,1)\halflength}%
    \advance\ypos by\halflength
    \advance\ypos by\gap
    \put(\xpos,\ypos){\evector(0,1){\halflength}}%
\fi\fi\fi\fi\fi\fi
}

\def\putmorphism(#1)(#2,#3)[#4`#5`#6]#7#8#9{{%
\run #2
\rise #3
\ifnum\rise=0
  \puthmorphism(#1)[#4`#5`#6]{#7}{#8}{#9}%
\else\ifnum\run=0
  \putvmorphism(#1)[#4`#5`#6]{#7}{#8}{#9}%
\else
\setpos(#1)%
\arrowlength #7
\arrowtype #8
\ifnum\run=0
\else\ifnum\rise=0
\else
\ifnum\run>0
    \coefa=1
\else
   \coefa=-1
\fi
\ifnum\arrowtype>0
   \coefb=0
   \coefc=-1
\else
   \coefb=\coefa
   \coefc=1
   \arrowtype=-\arrowtype
\fi
\width=2
\multiply \width by\run
\divide \width by\rise
\ifnum \width<0  \width=-\width\fi
\advance\width by60
\if l#9 \width=-\width\fi
\putbox(\xpos,\ypos){#4}
{\multiply \coefa by\arrowlength
\advance\xpos by\coefa
\multiply \coefa by\rise
\divide \coefa by\run
\advance \ypos by\coefa
\putbox(\xpos,\ypos){#5} }%
{\multiply \coefa by\arrowlength
\divide \coefa by2
\advance \xpos by\coefa
\advance \xpos by\width
\multiply \coefa by\rise
\divide \coefa by\run
\advance \ypos by\coefa
\if l#9%
   \put(\xpos,\ypos){\makebox(0,0)[r]{$#6$}}%
\else\if r#9%
   \put(\xpos,\ypos){\makebox(0,0)[l]{$#6$}}%
\fi\fi }%
{\multiply \rise by-\coefc
\multiply \run by-\coefc
\multiply \coefb by\arrowlength
\advance \xpos by\coefb
\multiply \coefb by\rise
\divide \coefb by\run
\advance \ypos by\coefb
\multiply \coefc by70
\advance \ypos by\coefc
\multiply \coefc by\run
\divide \coefc by\rise
\advance \xpos by\coefc
\multiply \coefa by140
\multiply \coefa by\run
\divide \coefa by\rise
\advance \arrowlength by\coefa
\ifnum \arrowtype=1
   \put(\xpos,\ypos){\vector(\run,\rise){\arrowlength}}%
\else\ifnum\arrowtype=2
   \put(\xpos,\ypos){\mvector(\run,\rise){\arrowlength}}%
\else\ifnum\arrowtype=3
   \put(\xpos,\ypos){\evector(\run,\rise){\arrowlength}}%
\fi\fi\fi}\fi\fi\fi\fi}}

\def\puthmorphism(#1,#2)[#3`#4`#5]#6#7#8{{%
\xpos #1
\ypos #2
\width #6
\arrowlength #6
\putbox(\xpos,\ypos){#3\vphantom{#4}}%
{\advance \xpos by\arrowlength
\putbox(\xpos,\ypos){\vphantom{#3}#4}}%
\horsize{\tempcounta}{#3}%
\horsize{\tempcountb}{#4}%
\divide \tempcounta by2
\divide \tempcountb by2
\advance \tempcounta by30
\advance \tempcountb by30
\advance \xpos by\tempcounta
\advance \arrowlength by-\tempcounta
\advance \arrowlength by-\tempcountb
\putvector(\xpos,\ypos)(1,0){\arrowlength}{#7}%
\divide \arrowlength by2
\advance \xpos by\arrowlength
\vertsize{\tempcounta}{#5}%
\divide\tempcounta by2
\advance \tempcounta by20
\if a#8 %
   \advance \ypos by\tempcounta
   \putbox(\xpos,\ypos){#5}%
\else
   \advance \ypos by-\tempcounta
   \putbox(\xpos,\ypos){#5}%
\fi}}

\def\putvmorphism(#1,#2)[#3`#4`#5]#6#7#8{{%
\xpos #1
\ypos #2
\arrowlength #6
\arrowtype #7
\settowidth{\xlen}{$#5$}%
\putbox(\xpos,\ypos){#3}%
{\advance \ypos by-\arrowlength
\putbox(\xpos,\ypos){#4}}%
{\advance\arrowlength by-140
\advance \ypos by-70
\ifdim\xlen>0pt
   \if m#8%
      \putsplitvector(\xpos,\ypos){\arrowlength}{\arrowtype}%
   \else
      \putvector(\xpos,\ypos)(0,-1){\arrowlength}{\arrowtype}%
   \fi
\else
   \putvector(\xpos,\ypos)(0,-1){\arrowlength}{\arrowtype}%
\fi}%
\ifdim\xlen>0pt
   \divide \arrowlength by2
   \advance\ypos by-\arrowlength
   \if l#8%
      \advance \xpos by-40
      \put(\xpos,\ypos){\makebox(0,0)[r]{$#5$}}%
   \else\if r#8%
      \advance \xpos by40
      \put(\xpos,\ypos){\makebox(0,0)[l]{$#5$}}%
   \else
      \putbox(\xpos,\ypos){#5}%
   \fi\fi
\fi
}}

\def\topadjust[#1`#2`#3]{%
\yoff=10
\vertadjust[#1`#2`{#3}]%
\advance \yext by\tempcounta
\advance \yext by 10
}
\def\botadjust[#1`#2`#3]{%
\vertadjust[#1`#2`{#3}]%
\advance \yext by\tempcounta
\advance \yoff by-\tempcounta
}
\def\leftadjust[#1`#2`#3]{%
\xoff=0
\horadjust[#1`#2`{#3}]%
\advance \xext by\tempcounta
\advance \xoff by-\tempcounta
}
\def\rightadjust[#1`#2`#3]{%
\horadjust[#1`#2`{#3}]%
\advance \xext by\tempcounta
}
\def\rightsladjust[#1`#2`#3]{%
\sladjust[#1`#2`{#3}]{\width}%
\advance \xext by\tempcounta
}
\def\leftsladjust[#1`#2`#3]{%
\xoff=0
\sladjust[#1`#2`{#3}]{\width}%
\advance \xext by\tempcounta
\advance \xoff by-\tempcounta
}
\def\adjust[#1`#2;#3`#4;#5`#6;#7`#8]{%
\topadjust[#1``{#2}]
\leftadjust[#3``{#4}]
\rightadjust[#5``{#6}]
\botadjust[#7``{#8}]}

\def\putsquarep<#1>(#2)[#3;#4`#5`#6`#7]{{%
\setsqparms[#1]%
\setpos(#2)%
\settokens[#3]%
\puthmorphism(\xpos,\ypos)[\tokenc`\tokend`{#7}]{\width}{\arrowtyped}b%
\advance\ypos by \height
\puthmorphism(\xpos,\ypos)[\tokena`\tokenb`{#4}]{\width}{\arrowtypea}a%
\putvmorphism(\xpos,\ypos)[``{#5}]{\height}{\arrowtypeb}l%
\advance\xpos by \width
\putvmorphism(\xpos,\ypos)[``{#6}]{\height}{\arrowtypec}r%
}}

\def\putsquare{\@ifnextchar <{\putsquarep}{\putsquarep%
   <\arrowtypea`\arrowtypeb`\arrowtypec`\arrowtyped;\width`\height>}}
\def\square{\@ifnextchar< {\squarep}{\squarep
   <\arrowtypea`\arrowtypeb`\arrowtypec`\arrowtyped;\width`\height>}}
\def\squarep<#1>[#2`#3`#4`#5;#6`#7`#8`#9]{{
\setsqparms[#1]
\xext=\width                                          
\yext=\height                                         
\topadjust[#2`#3`{#6}]
\botadjust[#4`#5`{#9}]
\leftadjust[#2`#4`{#7}]
\rightadjust[#3`#5`{#8}]
\begin{picture}(\xext,\yext)(\xoff,\yoff)
\putsquarep<\arrowtypea`\arrowtypeb`\arrowtypec`\arrowtyped;\width`\height>%
(0,0)[#2`#3`#4`#5;#6`#7`#8`{#9}]%
\end{picture}%
}}

\def\putptrianglep<#1>(#2,#3)[#4`#5`#6;#7`#8`#9]{{%
\settriparms[#1]%
\xpos=#2 \ypos=#3
\advance\ypos by \height
\puthmorphism(\xpos,\ypos)[#4`#5`{#7}]{\height}{\arrowtypea}a%
\putvmorphism(\xpos,\ypos)[`#6`{#8}]{\height}{\arrowtypeb}l%
\advance\xpos by\height
\putmorphism(\xpos,\ypos)(-1,-1)[``{#9}]{\height}{\arrowtypec}r%
}}

\def\putptriangle{\@ifnextchar <{\putptrianglep}{\putptrianglep
   <\arrowtypea`\arrowtypeb`\arrowtypec;\height>}}
\def\ptriangle{\@ifnextchar <{\ptrianglep}{\ptrianglep
   <\arrowtypea`\arrowtypeb`\arrowtypec;\height>}}

\def\ptrianglep<#1>[#2`#3`#4;#5`#6`#7]{{
\settriparms[#1]%
\width=\height                         
\xext=\width                           
\yext=\width                           
\topadjust[#2`#3`{#5}]
\botadjust[#3``]
\leftadjust[#2`#4`{#6}]
\rightsladjust[#3`#4`{#7}]
\begin{picture}(\xext,\yext)(\xoff,\yoff)
\putptrianglep<\arrowtypea`\arrowtypeb`\arrowtypec;\height>%
(0,0)[#2`#3`#4;#5`#6`{#7}]%
\end{picture}%
}}

\def\putqtrianglep<#1>(#2,#3)[#4`#5`#6;#7`#8`#9]{{%
\settriparms[#1]%
\xpos=#2 \ypos=#3
\advance\ypos by\height
\puthmorphism(\xpos,\ypos)[#4`#5`{#7}]{\height}{\arrowtypea}a%
\putmorphism(\xpos,\ypos)(1,-1)[``{#8}]{\height}{\arrowtypeb}l%
\advance\xpos by\height
\putvmorphism(\xpos,\ypos)[`#6`{#9}]{\height}{\arrowtypec}r%
}}

\def\putqtriangle{\@ifnextchar <{\putqtrianglep}{\putqtrianglep
   <\arrowtypea`\arrowtypeb`\arrowtypec;\height>}}
\def\qtriangle{\@ifnextchar <{\qtrianglep}{\qtrianglep
   <\arrowtypea`\arrowtypeb`\arrowtypec;\height>}}

\def\qtrianglep<#1>[#2`#3`#4;#5`#6`#7]{{
\settriparms[#1]
\width=\height                         
\xext=\width                           
\yext=\height                          
\topadjust[#2`#3`{#5}]
\botadjust[#4``]
\leftsladjust[#2`#4`{#6}]
\rightadjust[#3`#4`{#7}]
\begin{picture}(\xext,\yext)(\xoff,\yoff)
\putqtrianglep<\arrowtypea`\arrowtypeb`\arrowtypec;\height>%
(0,0)[#2`#3`#4;#5`#6`{#7}]%
\end{picture}%
}}

\def\putdtrianglep<#1>(#2,#3)[#4`#5`#6;#7`#8`#9]{{%
\settriparms[#1]%
\xpos=#2 \ypos=#3
\puthmorphism(\xpos,\ypos)[#5`#6`{#9}]{\height}{\arrowtypec}b%
\advance\xpos by \height \advance\ypos by\height
\putmorphism(\xpos,\ypos)(-1,-1)[``{#7}]{\height}{\arrowtypea}l%
\putvmorphism(\xpos,\ypos)[#4``{#8}]{\height}{\arrowtypeb}r%
}}

\def\putdtriangle{\@ifnextchar <{\putdtrianglep}{\putdtrianglep
   <\arrowtypea`\arrowtypeb`\arrowtypec;\height>}}
\def\dtriangle{\@ifnextchar <{\dtrianglep}{\dtrianglep
   <\arrowtypea`\arrowtypeb`\arrowtypec;\height>}}

\def\dtrianglep<#1>[#2`#3`#4;#5`#6`#7]{{
\settriparms[#1]
\width=\height                         
\xext=\width                           
\yext=\height                          
\topadjust[#2``]
\botadjust[#3`#4`{#7}]
\leftsladjust[#3`#2`{#5}]
\rightadjust[#2`#4`{#6}]
\begin{picture}(\xext,\yext)(\xoff,\yoff)
\putdtrianglep<\arrowtypea`\arrowtypeb`\arrowtypec;\height>%
(0,0)[#2`#3`#4;#5`#6`{#7}]%
\end{picture}%
}}

\def\putbtrianglep<#1>(#2,#3)[#4`#5`#6;#7`#8`#9]{{%
\settriparms[#1]%
\xpos=#2 \ypos=#3
\puthmorphism(\xpos,\ypos)[#5`#6`{#9}]{\height}{\arrowtypec}b%
\advance\ypos by\height
\putmorphism(\xpos,\ypos)(1,-1)[``{#8}]{\height}{\arrowtypeb}r%
\putvmorphism(\xpos,\ypos)[#4``{#7}]{\height}{\arrowtypea}l%
}}

\def\putbtriangle{\@ifnextchar <{\putbtrianglep}{\putbtrianglep
   <\arrowtypea`\arrowtypeb`\arrowtypec;\height>}}
\def\btriangle{\@ifnextchar <{\btrianglep}{\btrianglep
   <\arrowtypea`\arrowtypeb`\arrowtypec;\height>}}

\def\btrianglep<#1>[#2`#3`#4;#5`#6`#7]{{
\settriparms[#1]
\width=\height                         
\xext=\width                           
\yext=\height                          
\topadjust[#2``]
\botadjust[#3`#4`{#7}]
\leftadjust[#2`#3`{#5}]
\rightsladjust[#4`#2`{#6}]
\begin{picture}(\xext,\yext)(\xoff,\yoff)
\putbtrianglep<\arrowtypea`\arrowtypeb`\arrowtypec;\height>%
(0,0)[#2`#3`#4;#5`#6`{#7}]%
\end{picture}%
}}

\def\putAtrianglep<#1>(#2,#3)[#4`#5`#6;#7`#8`#9]{{%
\settriparms[#1]%
\xpos=#2 \ypos=#3
{\multiply \height by2
\puthmorphism(\xpos,\ypos)[#5`#6`{#9}]{\height}{\arrowtypec}b}%
\advance\xpos by\height \advance\ypos by\height
\putmorphism(\xpos,\ypos)(-1,-1)[#4``{#7}]{\height}{\arrowtypea}l%
\putmorphism(\xpos,\ypos)(1,-1)[``{#8}]{\height}{\arrowtypeb}r%
}}

\def\putAtriangle{\@ifnextchar <{\putAtrianglep}{\putAtrianglep
   <\arrowtypea`\arrowtypeb`\arrowtypec;\height>}}
\def\Atriangle{\@ifnextchar <{\Atrianglep}{\Atrianglep
   <\arrowtypea`\arrowtypeb`\arrowtypec;\height>}}

\def\Atrianglep<#1>[#2`#3`#4;#5`#6`#7]{{
\settriparms[#1]
\width=\height                         
\xext=\width                           
\yext=\height                          
\topadjust[#2``]
\botadjust[#3`#4`{#7}]
\multiply \xext by2 
\leftsladjust[#3`#2`{#5}]
\rightsladjust[#4`#2`{#6}]
\begin{picture}(\xext,\yext)(\xoff,\yoff)%
\putAtrianglep<\arrowtypea`\arrowtypeb`\arrowtypec;\height>%
(0,0)[#2`#3`#4;#5`#6`{#7}]%
\end{picture}%
}}

\def\putAtrianglepairp<#1>(#2)[#3;#4`#5`#6`#7`#8]{{
\settripairparms[#1]%
\setpos(#2)%
\settokens[#3]%
\puthmorphism(\xpos,\ypos)[\tokenb`\tokenc`{#7}]{\height}{\arrowtyped}b%
\advance\xpos by\height
\advance\ypos by\height
\putmorphism(\xpos,\ypos)(-1,-1)[\tokena``{#4}]{\height}{\arrowtypea}l%
\putvmorphism(\xpos,\ypos)[``{#5}]{\height}{\arrowtypeb}m%
\putmorphism(\xpos,\ypos)(1,-1)[``{#6}]{\height}{\arrowtypec}r%
}}

\def\putAtrianglepair{\@ifnextchar <{\putAtrianglepairp}{\putAtrianglepairp%
   <\arrowtypea`\arrowtypeb`\arrowtypec`\arrowtyped`\arrowtypee;\height>}}
\def\Atrianglepair{\@ifnextchar <{\Atrianglepairp}{\Atrianglepairp%
   <\arrowtypea`\arrowtypeb`\arrowtypec`\arrowtyped`\arrowtypee;\height>}}

\def\Atrianglepairp<#1>[#2;#3`#4`#5`#6`#7]{{%
\settripairparms[#1]%
\settokens[#2]%
\width=\height
\xext=\width
\yext=\height
\topadjust[\tokena``]%
\vertadjust[\tokenb`\tokenc`{#6}]
\tempcountd=\tempcounta                       
\vertadjust[\tokenc`\tokend`{#7}]
\ifnum\tempcounta<\tempcountd                 
\tempcounta=\tempcountd\fi                    
\advance \yext by\tempcounta                  
\advance \yoff by-\tempcounta                 %
\multiply \xext by2 
\leftsladjust[\tokenb`\tokena`{#3}]
\rightsladjust[\tokend`\tokena`{#5}]%
\begin{picture}(\xext,\yext)(\xoff,\yoff)%
\putAtrianglepairp
<\arrowtypea`\arrowtypeb`\arrowtypec`\arrowtyped`\arrowtypee;\height>%
(0,0)[#2;#3`#4`#5`#6`{#7}]%
\end{picture}%
}}

\def\putVtrianglep<#1>(#2,#3)[#4`#5`#6;#7`#8`#9]{{%
\settriparms[#1]%
\xpos=#2 \ypos=#3
\advance\ypos by\height
{\multiply\height by2
\puthmorphism(\xpos,\ypos)[#4`#5`{#7}]{\height}{\arrowtypea}a}%
\putmorphism(\xpos,\ypos)(1,-1)[`#6`{#8}]{\height}{\arrowtypeb}l%
\advance\xpos by\height
\advance\xpos by\height
\putmorphism(\xpos,\ypos)(-1,-1)[``{#9}]{\height}{\arrowtypec}r%
}}

\def\putVtriangle{\@ifnextchar <{\putVtrianglep}{\putVtrianglep
   <\arrowtypea`\arrowtypeb`\arrowtypec;\height>}}
\def\Vtriangle{\@ifnextchar <{\Vtrianglep}{\Vtrianglep
   <\arrowtypea`\arrowtypeb`\arrowtypec;\height>}}

\def\Vtrianglep<#1>[#2`#3`#4;#5`#6`#7]{{
\settriparms[#1]
\width=\height                         
\xext=\width                           
\yext=\height                          
\topadjust[#2`#3`{#5}]
\botadjust[#4``]
\multiply \xext by2 
\leftsladjust[#2`#3`{#6}]
\rightsladjust[#3`#4`{#7}]
\begin{picture}(\xext,\yext)(\xoff,\yoff)%
\putVtrianglep<\arrowtypea`\arrowtypeb`\arrowtypec;\height>%
(0,0)[#2`#3`#4;#5`#6`{#7}]%
\end{picture}%
}}

\def\putVtrianglepairp<#1>(#2)[#3;#4`#5`#6`#7`#8]{{
\settripairparms[#1]%
\setpos(#2)%
\settokens[#3]%
\advance\ypos by\height
\putmorphism(\xpos,\ypos)(1,-1)[`\tokend`{#6}]{\height}{\arrowtypec}l%
\puthmorphism(\xpos,\ypos)[\tokena`\tokenb`{#4}]{\height}{\arrowtypea}a%
\advance\xpos by\height
\putvmorphism(\xpos,\ypos)[``{#7}]{\height}{\arrowtyped}m%
\advance\xpos by\height
\putmorphism(\xpos,\ypos)(-1,-1)[``{#8}]{\height}{\arrowtypee}r%
}}

\def\putVtrianglepair{\@ifnextchar <{\putVtrianglepairp}{\putVtrianglepairp%
    <\arrowtypea`\arrowtypeb`\arrowtypec`\arrowtyped`\arrowtypee;\height>}}
\def\Vtrianglepair{\@ifnextchar <{\Vtrianglepairp}{\Vtrianglepairp%
    <\arrowtypea`\arrowtypeb`\arrowtypec`\arrowtyped`\arrowtypee;\height>}}

\def\Vtrianglepairp<#1>[#2;#3`#4`#5`#6`#7]{{%
\settripairparms[#1]%
\settokens[#2]
\xext=\height                  
\width=\height                 
\yext=\height                  
\vertadjust[\tokena`\tokenb`{#4}]
\tempcountd=\tempcounta        
\vertadjust[\tokenb`\tokenc`{#5}]
\ifnum\tempcounta<\tempcountd%
\tempcounta=\tempcountd\fi
\advance \yext by\tempcounta
\botadjust[\tokend``]%
\multiply \xext by2
\leftsladjust[\tokena`\tokend`{#6}]%
\rightsladjust[\tokenc`\tokend`{#7}]%
\begin{picture}(\xext,\yext)(\xoff,\yoff)%
\putVtrianglepairp
<\arrowtypea`\arrowtypeb`\arrowtypec`\arrowtyped`\arrowtypee;\height>%
(0,0)[#2;#3`#4`#5`#6`{#7}]%
\end{picture}%
}}

\def\putCtrianglep<#1>(#2,#3)[#4`#5`#6;#7`#8`#9]{{%
\settriparms[#1]%
\xpos=#2 \ypos=#3
\advance\ypos by\height
\putmorphism(\xpos,\ypos)(1,-1)[``{#9}]{\height}{\arrowtypec}l%
\advance\xpos by\height
\advance\ypos by\height
\putmorphism(\xpos,\ypos)(-1,-1)[#4`#5`{#7}]{\height}{\arrowtypea}l%
{\multiply\height by 2
\putvmorphism(\xpos,\ypos)[`#6`{#8}]{\height}{\arrowtypeb}r}%
}}

\def\putCtriangle{\@ifnextchar <{\putCtrianglep}{\putCtrianglep
    <\arrowtypea`\arrowtypeb`\arrowtypec;\height>}}
\def\Ctriangle{\@ifnextchar <{\Ctrianglep}{\Ctrianglep
    <\arrowtypea`\arrowtypeb`\arrowtypec;\height>}}

\def\Ctrianglep<#1>[#2`#3`#4;#5`#6`#7]{{
\settriparms[#1]
\width=\height                          
\xext=\width                            
\yext=\height                           
\multiply \yext by2 
\topadjust[#2``]
\botadjust[#4``]
\sladjust[#3`#2`{#5}]{\width}
\tempcountd=\tempcounta                 
\sladjust[#3`#4`{#7}]{\width}
\ifnum \tempcounta<\tempcountd          
\tempcounta=\tempcountd\fi              
\advance \xext by\tempcounta            
\advance \xoff by-\tempcounta           %
\rightadjust[#2`#4`{#6}]
\begin{picture}(\xext,\yext)(\xoff,\yoff)%
\putCtrianglep<\arrowtypea`\arrowtypeb`\arrowtypec;\height>%
(0,0)[#2`#3`#4;#5`#6`{#7}]%
\end{picture}%
}}

\def\putDtrianglep<#1>(#2,#3)[#4`#5`#6;#7`#8`#9]{{%
\settriparms[#1]%
\xpos=#2 \ypos=#3
\advance\xpos by\height \advance\ypos by\height
\putmorphism(\xpos,\ypos)(-1,-1)[``{#9}]{\height}{\arrowtypec}r%
\advance\xpos by-\height \advance\ypos by\height
\putmorphism(\xpos,\ypos)(1,-1)[`#5`{#8}]{\height}{\arrowtypeb}r%
{\multiply\height by 2
\putvmorphism(\xpos,\ypos)[#4`#6`{#7}]{\height}{\arrowtypea}l}%
}}

\def\putDtriangle{\@ifnextchar <{\putDtrianglep}{\putDtrianglep
    <\arrowtypea`\arrowtypeb`\arrowtypec;\height>}}
\def\Dtriangle{\@ifnextchar <{\Dtrianglep}{\Dtrianglep
   <\arrowtypea`\arrowtypeb`\arrowtypec;\height>}}

\def\Dtrianglep<#1>[#2`#3`#4;#5`#6`#7]{{
\settriparms[#1]
\width=\height                         
\xext=\height                          
\yext=\height                          
\multiply \yext by2 
\topadjust[#2``]
\botadjust[#4``]
\leftadjust[#2`#4`{#5}]
\sladjust[#3`#2`{#5}]{\height}
\tempcountd=\tempcountd                
\sladjust[#3`#4`{#7}]{\height}
\ifnum \tempcounta<\tempcountd         
\tempcounta=\tempcountd\fi             
\advance \xext by\tempcounta           %
\begin{picture}(\xext,\yext)(\xoff,\yoff)
\putDtrianglep<\arrowtypea`\arrowtypeb`\arrowtypec;\height>%
(0,0)[#2`#3`#4;#5`#6`{#7}]%
\end{picture}%
}}

\def\setrecparms[#1`#2]{\width=#1 \height=#2}%
%

\def\recursep<#1`#2>[#3;#4`#5`#6`#7`#8]{{%
\width=#1 \height=#2
\settokens[#3]
\settowidth{\tempdimen}{$\tokena$}
\ifdim\tempdimen=0pt
  \savebox{\tempboxa}{\hbox{$\tokenb$}}%
  \savebox{\tempboxb}{\hbox{$\tokend$}}%
  \savebox{\tempboxc}{\hbox{$#6$}}%
\else
  \savebox{\tempboxa}{\hbox{$\hbox{$\tokena$}\times\hbox{$\tokenb$}$}}%
  \savebox{\tempboxb}{\hbox{$\hbox{$\tokena$}\times\hbox{$\tokend$}$}}%
  \savebox{\tempboxc}{\hbox{$\hbox{$\tokena$}\times\hbox{$#6$}$}}%
\fi
\ypos=\height
\divide\ypos by 2
\xpos=\ypos
\advance\xpos by \width
\xext=\xpos \yext=\height
\topadjust[#3`\usebox{\tempboxa}`{#4}]%
\botadjust[#5`\usebox{\tempboxb}`{#8}]%
\sladjust[\tokenc`\tokenb`{#5}]{\ypos}%
\tempcountd=\tempcounta
\sladjust[\tokenc`\tokend`{#5}]{\ypos}%
\ifnum \tempcounta<\tempcountd
\tempcounta=\tempcountd\fi
\advance \xext by\tempcounta
\advance \xoff by-\tempcounta
\rightadjust[\usebox{\tempboxa}`\usebox{\tempboxb}`\usebox{\tempboxc}]%
\bfig
\putCtrianglep<-1`1`1;\ypos>(0,0)[`\tokenc`;#5`#6`{#7}]%
\puthmorphism(\ypos,0)[\tokend`\usebox{\tempboxb}`{#8}]{\width}{-1}b%
\puthmorphism(\ypos,\height)[\tokenb`\usebox{\tempboxa}`{#4}]{\width}{-1}a%
\advance\ypos by \width
\putvmorphism(\ypos,\height)[``\usebox{\tempboxc}]{\height}1r%
\efig
}}

\def\recurse{\@ifnextchar <{\recursep}{\recursep<\width`\height>}}

\def\puttwohmorphisms(#1,#2)[#3`#4;#5`#6]#7#8#9{{%
%
\puthmorphism(#1,#2)[#3`#4`]{#7}0a
\ypos=#2
\advance\ypos by 20
\puthmorphism(#1,\ypos)[\phantom{#3}`\phantom{#4}`#5]{#7}{#8}a
\advance\ypos by -40
\puthmorphism(#1,\ypos)[\phantom{#3}`\phantom{#4}`#6]{#7}{#9}b
}}

\def\puttwovmorphisms(#1,#2)[#3`#4;#5`#6]#7#8#9{{%
%
%
%
\putvmorphism(#1,#2)[#3`#4`]{#7}0a
\xpos=#1
\advance\xpos by -20
\putvmorphism(\xpos,#2)[\phantom{#3}`\phantom{#4}`#5]{#7}{#8}l
\advance\xpos by 40
\putvmorphism(\xpos,#2)[\phantom{#3}`\phantom{#4}`#6]{#7}{#9}r
}}

\def\puthcoequalizer(#1)[#2`#3`#4;#5`#6`#7]#8#9{{%
%
\setpos(#1)%
\puttwohmorphisms(\xpos,\ypos)[#2`#3;#5`#6]{#8}11%
\advance\xpos by #8
\puthmorphism(\xpos,\ypos)[\phantom{#3}`#4`#7]{#8}1{#9}
}}

\def\putvcoequalizer(#1)[#2`#3`#4;#5`#6`#7]#8#9{{%
%
%
%
%
\setpos(#1)%
\puttwovmorphisms(\xpos,\ypos)[#2`#3;#5`#6]{#8}11%
\advance\ypos by -#8
\putvmorphism(\xpos,\ypos)[\phantom{#3}`#4`#7]{#8}1{#9}
}}

\def\putthreehmorphisms(#1)[#2`#3;#4`#5`#6]#7(#8)#9{{%
\setpos(#1) \settypes(#8)
\if a#9 %
     \vertsize{\tempcounta}{#5}%
     \vertsize{\tempcountb}{#6}%
     \ifnum \tempcounta<\tempcountb \tempcounta=\tempcountb \fi
\else
     \vertsize{\tempcounta}{#4}%
     \vertsize{\tempcountb}{#5}%
     \ifnum \tempcounta<\tempcountb \tempcounta=\tempcountb \fi
\fi
\advance \tempcounta by 60
\puthmorphism(\xpos,\ypos)[#2`#3`#5]{#7}{\arrowtypeb}{#9}
\advance\ypos by \tempcounta
\puthmorphism(\xpos,\ypos)[\phantom{#2}`\phantom{#3}`#4]{#7}{\arrowtypea}{#9}
\advance\ypos by -\tempcounta \advance\ypos by -\tempcounta
\puthmorphism(\xpos,\ypos)[\phantom{#2}`\phantom{#3}`#6]{#7}{\arrowtypec}{#9}
}}

\def\putarc(#1,#2)[#3`#4`#5]#6#7#8{{%
\xpos #1
\ypos #2
\width #6
\arrowlength #6
\putbox(\xpos,\ypos){#3\vphantom{#4}}%
{\advance \xpos by\arrowlength
\putbox(\xpos,\ypos){\vphantom{#3}#4}}%
\horsize{\tempcounta}{#3}%
\horsize{\tempcountb}{#4}%
\divide \tempcounta by2
\divide \tempcountb by2
\advance \tempcounta by30
\advance \tempcountb by30
\advance \xpos by\tempcounta
\advance \arrowlength by-\tempcounta
\advance \arrowlength by-\tempcountb
\halflength=\arrowlength \divide\halflength by 2
\divide\arrowlength by 5
\put(\xpos,\ypos){\bezier{\arrowlength}(0,0)(50,50)(\halflength,50)}
\ifnum #7=-1 \put(\xpos,\ypos){\vector(-3,-2)0} \fi
\advance\xpos by \halflength
\put(\xpos,\ypos){\xpos=\halflength \advance\xpos by -50
   \bezier{\arrowlength}(0,50)(\xpos,50)(\halflength,0)}
\ifnum #7=1 {\advance \xpos by
   \halflength \put(\xpos,\ypos){\vector(3,-2)0}} \fi
\advance\ypos by 50
\vertsize{\tempcounta}{#5}%
\divide\tempcounta by2
\advance \tempcounta by20
\if a#8 %
   \advance \ypos by\tempcounta
   \putbox(\xpos,\ypos){#5}%
\else
   \advance \ypos by-\tempcounta
   \putbox(\xpos,\ypos){#5}%
\fi
}}

\makeatother

\sloppy

\newcommand{\nl}{\hspace{2cm}\\ }

\def\pos{\Diamond}
\def\diam{{\tiny\Diamond}}

\def\lc{\lceil}
\def\rc{\rceil}
\def\lf{\lfloor}
\def\rf{\rfloor}
\def\lk{\langle}
\def\rk{\rangle}
\def\blk{\dot{\langle\!\!\langle}}
\def\brk{\dot{\rangle\!\!\rangle}}

\newcommand{\pa}{\parallel}
\newcommand{\lra}{\longrightarrow}
\newcommand{\hra}{\hookrightarrow}
\newcommand{\hla}{\hookleftarrow}
\newcommand{\ra}{\rightarrow}
\newcommand{\la}{\leftarrow}
\newcommand{\lla}{\longleftarrow}
\newcommand{\da}{\downarrow}
\newcommand{\ua}{\uparrow}
\newcommand{\dA}{\downarrow\!\!\!^\bullet}
\newcommand{\uA}{\uparrow\!\!\!_\bullet}
\newcommand{\Da}{\Downarrow}
\newcommand{\DA}{\Downarrow\!\!\!^\bullet}
\newcommand{\UA}{\Uparrow\!\!\!_\bullet}
\newcommand{\Ua}{\Uparrow}
\newcommand{\Lra}{\Longrightarrow}
\newcommand{\Ra}{\Rightarrow}
\newcommand{\Lla}{\Longleftarrow}
\newcommand{\La}{\Leftarrow}
\newcommand{\nperp}{\perp\!\!\!\!\!\setminus\;\;}
\newcommand{\pq}{\preceq}

\def\phi{\varphi}
\def\ve{\varepsilon}
\def\o{{\omega}}

\def\bA{{\bf A}}
\def\bM{{\bf M}}
\def\bN{{\bf N}}
\def\bC{{\bf C}}
\def\bI{{\bf I}}
\def\bL{{\bf L}}
\def\bT{{\bf T}}
\def\bS{{\bf S}}
\def\bD{{\bf D}}
\def\bB{{\bf B}}
\def\bW{{\bf W}}
\def\bP{{\bf P}}
\def\bX{{\bf X}}
\def\bY{{\bf Y}}
\def\ba{{\bf a}}
\def\bb{{\bf b}}
\def\bc{{\bf c}}
\def\bd{{\bf d}}
\def\bh{{\bf h}}
\def\bi{{\bf i}}
\def\bj{{\bf j}}
\def\bk{{\bf k}}
\def\bm{{\bf m}}
\def\bn{{\bf n}}
\def\bp{{\bf p}}
\def\bq{{\bf q}}
\def\be{{\bf e}}
\def\br{{\bf r}}
\def\bi{{\bf i}}
\def\bs{{\bf s}}
\def\bt{{\bf t}}
\def\jeden{{\bf 1}}
\def\dwa{{\bf 2}}
\def\trzy{{\bf 3}}

\def\cB{{\cal B}}
\def\cA{{\cal A}}
\def\cC{{\cal C}}
\def\cD{{\cal D}}
\def\cE{{\cal E}}
\def\cF{{\cal F}}
\def\cG{{\cal G}}
\def\cI{{\cal I}}
\def\cJ{{\cal J}}
\def\cK{{\cal K}}
\def\cL{{\cal L}}
\def\cN{{\cal N}}
\def\cM{{\cal M}}
\def\cO{{\cal O}}
\def\cP{{\cal P}}
\def\cQ{{\cal Q}}
\def\cR{{\cal R}}
\def\cS{{\cal S}}
\def\cT{{\cal T}}
\def\cU{{\cal U}}
\def\cV{{\cal V}}
\def\cW{{\cal W}}
\def\cX{{\cal X}}
\def\cY{{\cal Y}}

\def\bMonCat{{\bf MonCat}}
\def\Mlcv{{\bf Malcev}}
\def\cMlcv{{\bf coMalcev}}
\def\Mnd{{\bf Mnd}}
\def\Cmd{{\bf Cmd}}
\def\Mon{{\bf Mon}}
\def\BMon{{\bf BMon}}
\def\SMon{{\bf SMon}}
\def\id{{\bf id}}
\def\Id{{\bf Id}}
\def\bCat{{{\bf Cat}}}
\def\bCatrc{{{\bf Cat}_{rc}}}
\def\bMonCatrc{{{\bf MonCat}_{rc}}}
\def\oC{{{\omega}Cat}}

\pagenumbering{arabic} \setcounter{page}{1}

\title{\bf\Large The Formal Theory of Monoidal Monads}

\author{  Marek Zawadowski\\
}

\maketitle
\begin{abstract} We give a 3-categorical, purely formal argument
explaining why on the category of Kleisli algebras for a lax monoidal monad, and dually
on the category of Eilenberg-Moore algebras for an oplax monoidal monad, we always have
a natural monoidal structures. The key observation is that the 2-category of lax monoidal
monads in any 2-category D with finite products is isomorphic to the 2-category of monoidal
objects with oplax morphisms in the 2-category of monads with lax morphisms in D.
As we explain at the end of the paper a similar phenomenon occurs in many other situations.
\end{abstract}

\section{Introduction}
It is well known, c.f. \cite{Day} p. 30, that the category of Kleisli algebras for
a monoidal monad carries a monoidal structure. Dually, the category of Eilenberg-Moore algebras for an
opmonoidal monad carries a monoidal structure, as well.  Theorem 7.2 of \cite{Mo}, considerably improved
this result and then Theorem 2.9 of \cite{McC}  gives a still stronger formulation putting this
result into $2$-categorical context.  Theorem 2.9 of \cite{McC} says that the $2$-category of monoidal
categories, oplax morphisms, and monoidal natural transformations admits Eilenberg-Moore objects.
The main goal of this paper is to put those considerations into $3$-categorical context. We show that
in fact any $2$-category $\Mon_{op}(\cD)$ of monoidal objects, oplax $1$-morphisms, and monoidal $2$-cells
constructed in any $2$-category $\cD$ with finite products and admitting Eilenberg-Moore objects,
admits itself Eilenberg-Moore objects. As we are more interested in lax monoidal monads, we will be dealing
with them and Kleisli objects and we will be only pointing out what it implies in the dual case of oplax monoidal monads
and Eilenberg-Moore objects. The proof of the main Theorem \ref{Kleisli_moncat} is simple and purely formal
based on the observation, Lemma \ref{comm_structures}, that the $2$-categorical structures of monoidal objects
and of monads commutes, if taken with appropriate $1$-cells. The name `Formal Category Theory' for such kind of study was
suggested by S. MacLane. It was first developed in \cite{Gray} and later in many other places as in \cite{St} for monads.

The author's main motivations for this paper is the study of structures like signatures, signatures with amalgamations,
symmetric signatures, polynomial and analytic functors, c.f. \cite{Z}. Each of these structures carries a monoidal structure and here
we separate the case when it is simple and exists for a very general reason, due to the fact that the symmetrization monad on multisorted
signatures is not only monoidal but it also has some additional properties. This additional properties giving rise to a monoidal structure
on the category of Eilenberg-Moore algebras will be presented in another paper.

The paper is organized as follows.
For the sake of completeness, in Section 2, we describe in detail why the $2$-categorical definition of the Kleisli objects, c.f. \cite{St},
gives all the data we expect and that it agrees with the usual Kleisli category when considered in $2$-category of categories $\bCat$.
To appreciate the construction even more, we organize the data so constructed into various cells in $4$-category $\bf 3CAT$
of $3$-categories, $3$-functors, pseudo $3$-natural transformations, pseudo $3$-modifications, and perturbations.
In particular, we show how real life situations may lead to perturbations.
In Section 3, we spell the definition of a monoidal category in a $2$-category with finite products of $0$-cells. Moreover, we state key
technical result (Lemma \ref{comm_structures}), explaining in what sense the monoidal and the monad structures commute.
Using this fact, we prove, in Section 4,  Theorem \ref{Kleisli_moncat} concerning the existence of Kleisli objects in
$2$-categories of monoidal objects in $2$-categories with finite products. We also present this
result in an even more abstract form, Theorem \ref{Kleisli_moncat1}, as a certain lifting property. In Section 5, we state
these result in the dual case concerning oplax monoidal monads and Eilenberg-Moore objects. Finally, in Section 6, we show that
such results also holds, if we replace monoidal objects by braided or symmetric monoidal objects or even by either monads or comonads,
proviso we keep the 'laxness' of these structures opposite to the 'laxness' of the monads involved in the definition of either
the Kleisli or the Eilenberg-Moore objects.

I would like to thank Stanis\l aw Szawiel for the useful discussions.

\section{The Kleisli and Eilenberg-Moore objects}
The contents of this section is well known, possibly with some minor exception. We spell the definitions in detail as we will be refering to them later.

In this section  $\cD$ is an arbitrary $2$-category. Recall that a {\em monad} in $\cD$ consists of an object $\cC$ of $\cD$, a $1$-endocell $\cS: \cC\ra \cC$,
two $2$-cells $\eta:1_\cC\ra \cS$ and $\mu: \cS^2\ra \cS$ so that $\mu\circ\eta_\cS = 1_\cS=\mu\circ\cS(\eta)$ and $\mu\circ(\mu_\cS)=\mu\circ\cS(\mu)$.

\subsection{The Kleisli objects}

An {\em oplax morphism of monads} is a pair $(F,\tau): (\cC,\cS,\eta,\mu)\ra (\cC',\cS',\eta',\mu')$
such that $F:\cC\ra \cC'$ is a $1$-cell and $\tau : F\cS \ra \cS'F$ is a $2$-cell so that the diagram
\begin{center} \xext=1100 \yext=650
\begin{picture}(\xext,\yext)(\xoff,\yoff)
\setsqparms[-1`1`1`-1;600`400]
 \putsquare(470,50)[F\cS`F\cS^2`\cS'F`\cS'^2F;F(\mu)`\tau`\cS'(\tau)\circ\tau_{\cS}`\mu'_F]
 \put(120,300){\vector(2,1){250}}
  \put(120,200){\vector(2,-1){250}}
 \put(110,30){$\eta'_F$}
  \put(70,400){$F(\eta)$}
 \put(50,210){$F$}
 \end{picture}
\end{center}
commutes. The composition of two composable oplax morphisms of monads is given by $(F',\tau')\circ (F,\tau)=(F'\circ F, \tau'_F\circ F'(\tau))$.
A {\em transformation} $\sigma : (F,\tau)\ra (F',\tau')$ of two (parallel) oplax morphisms of monads is a $2$-cell $\sigma : F\ra F'$ making
the square
\begin{center} \xext=600 \yext=550
\begin{picture}(\xext,\yext)(\xoff,\yoff)
\setsqparms[1`1`1`1;600`400]
 \putsquare(0,50)[F\cS`F'\cS`\cS'F`\cS'F';\sigma_\cS`\tau`\tau'`S'(\sigma)]
 \end{picture}
\end{center}
commute. This defines the $2$-category $\Mnd_{op}(\cD)$ of monads in $\cD$ with oplax morphisms and transformations of oplax morphisms.
$\Mnd_{op}$ is a $3$-endofunctor on the $3$-category of $2$-categories $2\bCat$. On $1$- $2$- $3$-cells $\Mnd_{op}$ is defined in the obvious way.
We have an embedding $2$-functor $\iota_{op,\cD} : \cD \ra \Mnd_{op}(\cD)$ sending an object $\cC$ of $\cD$ to the identity monad on $\cC$. We often abbreviate $\iota_{op,\cD}$ to $\iota_{op}$.
$\iota_{op}$ has always a right $2$-adjoint $|-|=|-|_\cD$ sending a monad to its underlying category.
If  $\iota_{op}$ has a left $2$-adjoint $\cK=\cK_\cD$ we say, c.f. \cite{St}, that $\cD$ admits Kleisli objects.
\begin{center} \xext=1000 \yext=350
\begin{picture}(\xext,\yext)(\xoff,\yoff)
\putmorphism(0,150)(1,0)[\Mnd_{op}(\cD)`\cD`\iota_{op}]{1000}{-1}b

 \putmorphism(0,50)(1,0)[\phantom{\Mnd_{op}(\cD)}`\phantom{\cD}`{|-|}]{1000}{1}b
 \putmorphism(0,250)(1,0)[\phantom{\Mnd_{op}(\cD)}`\phantom{\cD}`{\cal K}]{1000}{1}a

 \end{picture}
\end{center}

If $H:\cD\ra \cD'$ is a $2$-functor between two $2$-categories that admit Kleisli objects, then we say that $H$ preserves Kleisli objects if
the canonical $2$-natural transformation in the square
\begin{center} \xext=1200 \yext=600
\begin{picture}(\xext,\yext)(\xoff,\yoff)
\setsqparms[1`1`1`1;1200`500]
 \putsquare(0,50)[\Mnd_{op}(\cD)`\Mnd_{op}(\cD')`\cD`\cD';\Mnd_{op}(H)`\cK_\cD`\cK_{\cD'}`H]
 \end{picture}
\end{center}
is a $2$-natural isomorphism.

\subsection{The Eilenberg-Moore objects}
A {\em lax morphism of monads} is a pair $(F,\tau): (\cC,\cS,\eta,\mu)\ra (\cC',\cS',\eta',\mu')$
such that $F:\cC\ra \cC'$ is a $1$-cell and $\tau :\cS'F  \ra F\cS$ is a $2$-cell so that the diagram
\begin{center} \xext=1100 \yext=550
\begin{picture}(\xext,\yext)(\xoff,\yoff)
\setsqparms[-1`1`1`-1;600`400]
 \putsquare(420,50)[\cS'F`\cS'^2F`F\cS`F\cS^2;\mu'_F`\tau`\tau_{\cS}\circ\cS'(\tau) `F(\mu)]
 \put(70,300){\vector(2,1){250}}
  \put(70,200){\vector(2,-1){250}}
 \put(60,30){$F(\eta)$}
  \put(20,400){$\eta'_F$}
 \put(0,210){$F$}
 \end{picture}
\end{center}
commutes. The composition of two composable lax morphisms of monads is given by $(F',\tau')\circ (F,\tau)=(F'\circ F, F'(\tau)\circ \tau'_F)$.
A {\em transformation} $\sigma : (F,\tau)\ra (F',\tau')$ of two (parallel) lax morphisms of monads is a $2$-cell $\sigma : F\ra F'$ making
the square
\begin{center} \xext=600 \yext=530
\begin{picture}(\xext,\yext)(\xoff,\yoff)
\setsqparms[1`1`1`1;600`400]
 \putsquare(0,30)[\cS'F`\cS'F'`F\cS`F'\cS;S'(\sigma)`\tau`\tau'`\sigma_\cS]
 \end{picture}
\end{center}
commute. This defines the $2$-category $\Mnd(\cD)$ of monads in $\cD$ with lax morphisms and transformations of lax morphisms.
$\Mnd$ is a $3$-endofunctor on the $3$-category of $2$-categories $2\bCat$.
We have an embedding $2$-functor $\iota_{\cD} : \cD \ra \Mnd(\cD)$  sending an object $\cC$ of $\cD$ to the identity monad on $\cC$.
We often abbreviate $\iota_{\cD}$ to $\iota$.
It has always a left $2$-adjoint $|-|=|-|_\cD$ sending a monad to its underlying category.
If  $\iota$ has a right $2$-adjoint $EM=EM_\cD$ we say, c.f. \cite{St}, that $\cD$ admits Eilenberg-Moore objects or EM objects.

\begin{center} \xext=1000 \yext=300
\begin{picture}(\xext,\yext)(\xoff,\yoff)
\putmorphism(0,150)(1,0)[\cD`\Mnd(\cD)`\iota]{1000}{1}b

 \putmorphism(0,50)(1,0)[\phantom{\cD}`\phantom{\Mnd(\cD)}`{EM}]{1000}{-1}b
 \putmorphism(0,250)(1,0)[\phantom{\cD}`\phantom{\Mnd(\cD)}`{|-|}]{1000}{-1}a
 \end{picture}
\end{center}
The preservation of EM objects is defined in the same way as the preservation of Kleisli objects.

\subsection{Some $3$-categories and $3$-functors}\label{some3cat}
$2\bCat$ is the $3$-category of $2$-categories, i.e. with $2$-categories as $0$-cells, $2$-functors as $1$-cells,
$2$-natural transformations as $2$-cells, and $2$-modifications as $3$-cells.

By a $2$-category with finite products, we will always mean a $2$-category with finite products of $0$-cells.
Let $2\bCat_\times$ be the sub-$3$-category of $2\bCat$ full on $2$-transformations and $2$-modifications,
whose $0$-cells are $2$-categories with finite products, and $1$-cells are $2$-functors  preserving
finite products.

Let $2\bCat_k$ be the sub-$3$-category of $2\bCat$ full on $2$-transformations and $2$-modifications,
whose $0$-cells are $2$-categories that admit Kleisli objects, and $1$-cells are $2$-functors  preserving
Kleisli objects.

Let $2\bCat_{em}$ be the sub-$3$-category of $2\bCat$ full on $2$-transformations and $2$-modifications,
whose $0$-cells are $2$-categories that admit EM objects, and $1$-cells are $2$-functors  preserving
EM objects.

These properties can be combined together. For example $2\bCat_{kem\times}$ is the
sub-$3$-category of $2\bCat$ full on $2$-transformations and $2$-modifications,
that admit all the mentioned constructions.

As we already mentioned, we have $3$-functors
\[  \Mnd, \Mnd_{op} : 2\bCat \lra 2\bCat \]
and these functors restrict to $3$-functors
\[  \Mnd_\times, \Mnd_{op,\times} : 2\bCat_\times \lra 2\bCat_\times \]
To see this, note that in the $2$-category $\cD$ with finite products,
the product of the monads $(C,\cS,\eta,\mu)$ and $(C',\cS',\eta',\mu')$ is,
the monad $(C\times C',\cS\times\cS',(\eta,\eta'),(\mu,\mu'))$.

\subsection{The $2$-categorical description of the Kleisli objects}
We describe below the above $3$-categorical definition of the Kleisli objects in $2$-categorical terms.

Thus we have $2$-adjunctions $\cK\dashv \iota_{op}\dashv |-|$. Let us fix a monad $(\cC,\cS,\eta,\mu)$ in $\cD$.
We will often abbreviate it to $\cS$. The unit of the adjunction $\iota_{op}\dashv |-|$ on $\cC$ is the identity
$1_\cC:\cC\ra |\iota_{op}(\cC)|$. The counit of this adjunction on $\cS$ is $(1_\cC,\eta): \iota_{op}|\cS| \ra \cS$.

The unit of the $2$-adjunction $\cK\dashv \iota_{op}$ on $\cS$ is the morphism adjoint to $1_{\cK(\cS)}$
\begin{center} \xext=1700 \yext=350
\begin{picture}(\xext,\yext)(\xoff,\yoff)
\putmorphism(250,250)(1,0)[\cK(\cS)`\cK(\cS)`1_{\cK(\cS)}]{1000}{1}a
\put(0,160){\line(1,0){1500}}
\put(1500,120){\makebox(300,100){$\dashv$}}

\putmorphism(250,70)(1,0)[\cS`\iota_{op}\cK(\cS)=1_{\cC_\cS}`(F_\cS,\kappa)]{1000}{1}b
\end{picture}
\end{center}
Thus $\cC_\cS$ is a $0$-cell in $\cD$, $F_\cS:\cC\ra \cC_\cS$ is a $1$-cell in $\cD$, and $\kappa: F_\cS\circ\cS\ra F_\cS$ is a $2$-cell in $\cD$ so that
in the diagram
\begin{center} \xext=1600 \yext=250
\begin{picture}(\xext,\yext)(\xoff,\yoff)
   \putmorphism(0,100)(1,0)[F_\cS\circ \cS^2`F_\cS\circ \cS`]{800}{0}a
   \putmorphism(0,50)(1,0)[\phantom{F_\cS\circ \cS^2}`\phantom{F_\cS\circ \cS}`F_\cS(\mu)]{800}{1}b
   \putmorphism(0,150)(1,0)[\phantom{F_\cS\circ \cS^2}`\phantom{F_\cS\circ \cS}`\kappa_{\cS}]{800}{1}a

  \putmorphism(800,100)(1,0)[\phantom{F_\cS\circ \cS}`F_\cS`]{800}{0}a
  \putmorphism(800,150)(1,0)[\phantom{F_\cS\circ \cS}`\phantom{F_\cS}`\kappa]{800}{1}a
  \putmorphism(800,50)(1,0)[\phantom{F_\cS\circ \cS}`\phantom{F_\cS}`F_\cS(\eta)]{800}{-1}b
\end{picture}
\end{center}
we have $\kappa \circ F_\cS=1_{F_\cS}$ and $\kappa\circ F_\cS(\mu)=\kappa\circ(\kappa_{F_\cS})$. In such circumstances we say that
$(F_\cS,\kappa)$ {\em subcoequalizes}  $\cS$. The counit of this adjunction on $\cC$ is $1_\cC : \cC=\cK\iota_{op}(\cC)\ra \cC$.

One can check directly that $(\cS,\mu):\cS\ra \iota_{op}(|\cS|)=1_\cC$ is an oplax  morphism of monads.
By adjunction
\begin{center} \xext=1700 \yext=350
\begin{picture}(\xext,\yext)(\xoff,\yoff)
\putmorphism(250,260)(1,0)[\cS`\iota_{op}|\cS|`(\cS,\mu)]{1000}{1}a
\put(0,170){\line(1,0){1500}}
\put(1500,110){\makebox(300,100){$\dashv$}}

\putmorphism(250,50)(1,0)[\cK(\cS)`|\cS|`U_\cS]{1000}{1}b
\end{picture}
\end{center}
we get the $1$-cell $U_\cS$.
Using twice the adjunction $\cK\dashv \iota_{op}$ we obtain
\begin{center} \xext=2600 \yext=600
\begin{picture}(\xext,\yext)(\xoff,\yoff)
\putmorphism(250,500)(1,0)[\cS`\iota_{op}\cK(\cS)`(F_\cS,\kappa)]{1000}{1}a
\putmorphism(1250,500)(1,0)[\phantom{\iota_{op}\cK(\cS)}`\iota_{op}|\cS|`\iota_{op}(U_\cS)]{1000}{1}a

\put(0,410){\line(1,0){2520}}
\put(2530,370){\makebox(300,100){$\dashv $}}

\putmorphism(250,250)(1,0)[\cK(\cS)`\cK(\cS)`1_{\cK(\cS)}]{1000}{1}a
\putmorphism(1250,250)(1,0)[\phantom{\cK(\cS)}`|\cS|`U_\cS]{1000}{1}a

\put(0,160){\line(1,0){2520}}
\put(2530,120){\makebox(300,100){$\dashv$}}

\putmorphism(750,0)(1,0)[\cS`\iota_{op}|\cS|`(\cS,\mu)]{1000}{1}a
\end{picture}
\end{center}
and by the uniqueness of adjoints, we get $(\cS,\mu)=\iota_{op}(U_\cS)\circ (F_\cS,\kappa)=(U_\cS F_\cS,U_\cS(\kappa))$.

The unit of the adjunction $F_\cS\dashv U_\cS$ in $\cD$ is $\eta$.
In order to define $\varepsilon$, the counit of this adjunction, we proceed as follows. First note that
we have equalities of oplax morphism of monads from $\cS$  to $\iota_{op}\cK(\cS)=1_{\cC_\cS}$:
$$(F_\cS\circ\cS,F_\cS(\mu))=\iota_{op}(|F_\cS,\kappa|)\circ (\cS,\mu)=\iota_{op}(F_\cS)\circ \iota_{op}(U_\cS)\circ (F_\cS,\kappa)$$
Note that the codomains of the morphisms are correct as $\iota_{op}|\iota_{op}\cK(\cS)| =\iota_{op}\cK(\cS)$.
The above morphism is parallel to $(\cS,\mu)$.
Since $\kappa\circ F_\cS(\mu)=\kappa\circ(\kappa_{F_\cS})$ it follows that
$$\kappa: (F_\cS\circ \cS, F_\cS(\mu)) \ra (F_\cS,\kappa)$$
is a transformation of oplax morphisms of monads, i.e. a $2$-cell in $\Mnd_{op}(\cD)$.
The adjoint correspondences of the $2$-cells below defines the counit $\varepsilon$:
\begin{center} \xext=2600 \yext=850
\begin{picture}(\xext,\yext)(\xoff,\yoff)
\putmorphism(250,770)(1,0)[\phantom{\cS}`\iota_{op}\cK(\cS)`(F_\cS,\kappa)]{800}{1}a
\putmorphism(1050,770)(1,0)[\phantom{\iota_{op}\cK(\cS)}`\iota_{op}|\cS|`\iota_{op}(U_\cS)]{800}{1}a
\putmorphism(1850,770)(1,0)[\phantom{\iota_{op}|\cS|}`\phantom{\iota_{op}\cK(\cS)}`\iota_{op}(F_\cS)]{800}{1}a
\putmorphism(250,550)(1,0)[\phantom{\cS}`\phantom{\iota_{op}\cK(\cS)}`(F_\cS,\kappa)]{2400}{1}b
\putmorphism(250,650)(1,0)[\cS`\iota_{op}\cK(\cS)`]{2400}{0}b

\put(1400,600){\makebox(300,100){$\kappa\,\Da$}}

\put(0,410){\line(1,0){3020}}
\put(2930,370){\makebox(300,100){$\dashv $}}

\putmorphism(650,250)(1,0)[\phantom{\cK(\cS)}`|\cS|`U_\cS]{800}{1}a
\putmorphism(1450,250)(1,0)[\phantom{|\cS|}`\phantom{\cK(\cS)}`F_\cS]{800}{1}a

\put(1600,100){\makebox(300,100){$\varepsilon\,\Da$}}

\putmorphism(650,50)(1,0)[\phantom{\cK(\cS)}`\phantom{\cK(\cS)}`1_{\cK(\cS)}]{1600}{1}b
\putmorphism(650,150)(1,0)[\cK(\cS)`\cK(\cS)`]{1600}{0}b
\end{picture}
\end{center}
 We note for the record that $\varepsilon_{F_\cS}=\kappa$. Next, we verify the triangular equalities.
We have
\[ \varepsilon_{F_\cS}\circ F_\cS(\eta)=\kappa\circ F_\cS(\eta)= 1_{F_\cS} \]
The last equality follows from the fact that $(F_\cS,\kappa):\cS\ra 1_{\cC_\cS}$ is an oplax morphism
monads, i.e. $(F_\cS,\kappa)$ subequalizes $\cS$.

To see the other triangular equality, we consider the following correspondences of $2$-cells

\begin{center} \xext=3000 \yext=3200
\begin{picture}(\xext,\yext)(\xoff,\yoff)
\putmorphism(300,2620)(1,0)[\cK(\cS)`|\cS|`U_\cS]{800}{1}a
\putmorphism(1100,2620)(1,0)[\phantom{\cS|}`\cK(\cS)`F_\cS]{800}{1}a
\putmorphism(1900,2620)(1,0)[\phantom{\cK(\cS)}`|\cS|`U_\cS]{800}{1}a

\put(1200,2400){\makebox(300,100){$\varepsilon\;\Da$}}
\put(1850,2760){\makebox(300,100){$\eta\;\Da$}}

\put(250,2350){\line(1,0){1650}}
\put(250,2350){\line(0,1){200}}
\put(1900,2350){\vector(0,1){200}}
\put(1100,2200){\makebox(100,100){$1_{\cK_\cS}$}}
\put(1100,2950){\line(1,0){1600}}
\put(1100,2700){\line(0,1){250}}
\put(2700,2950){\vector(0,-1){250}}
\put(2000,3000){\makebox(100,100){$1_{|\cS|}$}}

\put(0,2160){\line(1,0){3100}}
\put(3100,2120){\makebox(100,100){$\dashv $}}

\putmorphism(100,1620)(1,0)[\cS`\iota_{op}\cK(\cS)`(F_\cS,\kappa)]{700}{1}a
\putmorphism(800,1620)(1,0)[\phantom{\iota_{op}\cK(\cS)}`\iota_{op}|\cS|`\iota_{op}(U_\cS)]{700}{1}a
\putmorphism(1500,1620)(1,0)[\phantom{\iota_{op}|\cS|}`\iota_{op}\cK(\cS)`\iota_{op}(F_\cS)]{700}{1}a
\putmorphism(2200,1620)(1,0)[\phantom{\iota_{op}\cK(\cS)}`\iota_{op}|\cS|`\iota_{op}(U_\cS)]{700}{1}a

\put(1000,1400){\makebox(300,100){$\kappa=\varepsilon_{F_\cS}\;\Da$}}
\put(2050,1760){\makebox(300,100){$\iota_{op}(\eta)\;\Da$}}

\put(100,1350){\line(1,0){2100}}
\put(100,1350){\line(0,1){200}}
\put(2200,1350){\vector(0,1){200}}
\put(1100,1200){\makebox(100,100){$(F_\cS,\kappa)$}}
\put(1450,1950){\line(1,0){1400}}
\put(1450,1700){\line(0,1){250}}
\put(2850,1950){\vector(0,-1){250}}
\put(2000,2000){\makebox(100,100){$\iota_{op}(1_{|\cS|})$}}

\put(0,1110){\line(1,0){3100}}
\put(3100,1070){\makebox(100,100){$=$}}

\putmorphism(550,950)(1,0)[\phantom{\cK(\cS)}`\phantom{|\cS|}`U_\cS]{1800}{1}a
\putmorphism(550,750)(1,0)[\cK(\cS)`|\cS|`]{1800}{1}a
\putmorphism(550,550)(1,0)[\phantom{\cK(\cS)}`\phantom{|\cS|}`U_\cS]{1800}{1}b

\put(1000,800){\makebox(300,100){$\iota_{op}(\eta_\cS)\;\Da$}}
\put(1000,600){\makebox(300,100){$U_\cS(\kappa)\;\Da$}}
\put(1700,770){\makebox(300,100){$(\cS^2,\cS(\mu))$}}

\put(0,410){\line(1,0){3100}}
\put(3100,370){\makebox(100,100){$\dashv $}}

\putmorphism(550,250)(1,0)[\phantom{\cK(\cS)}`\phantom{|\cS|}`U_\cS]{1800}{1}a
\putmorphism(550,50)(1,0)[\phantom{\cK(\cS)}`\phantom{|\cS|}`U_\cS]{1800}{1}b
\putmorphism(550,150)(1,0)[\cK(\cS)`|\cS|`]{1800}{0}b

\put(1200,100){\makebox(300,100){$1_{U_\cS}\;\Da$}}

\end{picture}
\end{center}
The first and the last are adjoint correspondences. In the middle, we have equality of $2$-cells.
The last $2$-cell is $1_{U_\cS}$ since before last is
$$U_{\cS}(\varepsilon_{F_\cS})\circ \iota_{op}(\eta_\cS)= U_{\cS}(\kappa)\circ \eta_\cS=  \mu\circ \eta_\cS= 1_{(S,\mu)} $$
This ends the $2$-categorical explanation why $\cK$ 'produces' the Kleisli object, if they exist.
The categorical explanation will be given in Subsection \ref{cat_description}.

\subsection{The $4$-categorical perspective}

We bring here some order to the data constructed above by describing it as some cells in the $4$-category
$3\bCat$ of (strict) $3$-categories, $3$-functors, pseudo-natural 3-transformations, pseudo 3-modifications, and perturbations.

We need some notation to be used only in the remainder of this subsection.
For a monad $\cS=(\cC,\cS,\eta,\mu)$ in a $2$-category $\cD$ the unit $\eta$
(and all other constructs derived from the monad $\cS$) will be denoted with a subscript $[\cD,\cS]$. Thus we write $\cC_{[\cD,\cS]}$ for $\cC$,
$\eta_{[\cD,\cS]}$ for the unit $\eta$, $\varepsilon_{[\cD,\cS]}$ for the counit $\varepsilon$ of the adjunction $F_\cS\dashv U_\cS$, i.e.
$F_{[\cD,\cS]}\dashv U_{[\cD,\cS]}$, and so on.

We have a modification $U$:
\begin{center} \xext=1600 \yext=450
\begin{picture}(\xext,\yext)(\xoff,\yoff)
\putmorphism(0,350)(1,0)[\phantom{2\bCat_k}`\phantom{2\bCat}`\Mnd_{op}]{1600}{1}a
\putmorphism(0,200)(1,0)[2\bCat_k`2\bCat`]{1600}{0}a
\putmorphism(0,50)(1,0)[\phantom{2\bCat_k}`\phantom{2\bCat}`Emb]{1600}{1}b

\put(1000,100){\makebox(300,100){$\Da\; |-|$}}
\put(1020,180){\line(0,1){120}}
\put(1040,180){\line(0,1){120}}

\put(300,100){\makebox(300,100){$\cK\;\Da$}}
\put(505,180){\line(0,1){120}}
\put(525,180){\line(0,1){120}}

\put(820,100){\makebox(50,100){$>$}}
\put(650,140){\line(1,0){200}}
\put(650,150){\line(1,0){220}}
\put(650,160){\line(1,0){200}}
\put(670,190){\makebox(150,100){$U$}}

\end{picture}
\end{center}
The $3$-functor $\Mnd_{op}$ is defined above, $Emb$ is the obvious embedding $3$-functor.
$|-|: \Mnd_{op}\ra Emb$ is a (strict) $3$-transformation so that $|-|_\cD:\Mnd_{op}(\cD)\ra \cD$ is
associating to a monad $\cS$ in $\cD$, its underlying category $|\cS|_\cD=\cC_{[\cD,\cS]}$.
$\cK : \Mnd_{op}\ra Emb$ is a (pseudo) $3$-transformation so that $\cK_\cD:\Mnd_{op}(\cD)\ra \cD$ is associating to a monad $\cS$ in $\cD$
its Kleisli category $\cK_\cD(\cS)$. The component  $U_\cD :\cK_\cD \ra |-|_\cD$  of the modification $U :\cK \ra |-|$  at $\cD$ is a
$2$-transformation of $2$-functors such that at the monad $S$ it is $U_{[\cD,\cS]} :\cK_\cD(\cS) \ra |\cS|_\cD$, i.e. the forgetful $1$-cell in $\cD$
from the Kleisli object for $\cS$ to the underlying category of $\cS$.

We also have a modification $F$:
\begin{center} \xext=1700 \yext=450
\begin{picture}(\xext,\yext)(\xoff,\yoff)
\putmorphism(0,350)(1,0)[\phantom{2\bCat_k}`\phantom{2\bCat}`\Mnd_{op}]{1700}{1}a
\putmorphism(0,200)(1,0)[2\bCat_k`2\bCat`]{1700}{0}a
\putmorphism(0,50)(1,0)[\phantom{2\bCat_k}`\phantom{2\bCat}`\Mnd_{op}]{1700}{1}b

\put(1100,100){\makebox(300,100){$\Da\; \iota_{op}\circ\cK$}}
\put(1066,180){\line(0,1){120}}
\put(1086,180){\line(0,1){120}}

\put(300,100){\makebox(300,100){$id_{\Mnd_{op}}\;\Da$}}
\put(644,180){\line(0,1){120}}
\put(664,180){\line(0,1){120}}

\put(920,100){\makebox(50,100){$>$}}
\put(750,140){\line(1,0){200}}
\put(750,150){\line(1,0){220}}
\put(750,160){\line(1,0){200}}
\put(770,190){\makebox(150,100){$F$}}

\end{picture}
\end{center}
$\iota_{op}\circ\cK: \Mnd_{op}\ra \Mnd_{op}$ is a (pseudo) $3$-transformation so that $\iota_{op,\cD}\circ\cK_\cD:\Mnd_{op}(\cD)\ra \Mnd_{op}(\cD)$ is
associating to a monad $\cS$ in $\cD$ the identity monad on $\cK_\cD(\cS)$, i.e. $\iota_{op,\cD}\circ\cK_\cD(\cS)= 1_{\cK_\cD(\cS)}$.

The component  $F_\cD : Id_{\Mnd_{op}(\cD)} \ra \iota_{op,\cD}\circ\cK_\cD$  of the modification $F : id_{\Mnd_{op}} \ra \iota_{op}\circ\cK$  at $\cD$ is a
$2$-transformation of $2$-functors such that at the monad $S$ it is
$(F_{[\cD,\cS]},\kappa_{[\cD,\cS]}) :\cS \ra 1_{\cK_\cD(\cS)}$. In particular $F_\cS=F_{[\cD,\cS]}=|(F_{[\cD,\cS]},\kappa_{[\cD,\cS]})|$
is the free Kleisli algebra  $1$-cell in $\cD$ from the underlying category of $\cS$ to the Kleisli object for $\cS$.

Now if we compose the $3$-transformation $|-|$ with the $3$-modification $F$ we get a  $3$-modification
\[ |F| : |-| \lra |-|\circ \iota_{op}\circ \cK= \cK   \]
Thus we can compose the $3$-modifications $|F|$ and $U$ both ways.
The perturbation $\eta$ (i.e. a $4$-cell in the $4$-category $3\bCat$)
from $Id_{|-|}$ to $U\circ |F|$ is described below. The following diagram
\begin{center} \xext=1700 \yext=1050
\begin{picture}(\xext,\yext)(\xoff,\yoff)
\putmorphism(0,950)(1,0)[\phantom{2\bCat_k}`\phantom{2\bCat}`\Mnd_{op}]{1700}{1}a
\putmorphism(0,500)(1,0)[2\bCat_k`2\bCat`]{1700}{0}a
\putmorphism(0,50)(1,0)[\phantom{2\bCat_k}`\phantom{2\bCat}`Emb]{1700}{1}b

\put(400,100){\makebox(50,100){$\Da$}}
\put(415,190){\line(0,1){620}}
\put(435,190){\line(0,1){620}}
\put(250,600){\makebox(50,100){$|-|$}}

\put(1300,100){\makebox(50,100){$\Da$}}
\put(1315,190){\line(0,1){620}}
\put(1335,190){\line(0,1){620}}
\put(1420,600){\makebox(50,100){$|-|$}}

\put(1070,200){\makebox(50,100){$>$}}
\put(600,240){\line(1,0){500}}
\put(600,250){\line(1,0){520}}
\put(600,260){\line(1,0){500}}
\put(800,100){\makebox(150,100){$U\circ |F|$}}

\put(1070,600){\makebox(50,100){$>$}}
\put(600,640){\line(1,0){500}}
\put(600,650){\line(1,0){520}}
\put(600,660){\line(1,0){500}}
\put(800,690){\makebox(150,100){$Id_{|-|}$}}

\put(847,330){\makebox(50,100){$\bigvee$}}
\put(865,345){\line(0,1){220}}
\put(875,345){\line(0,1){220}}
\put(855,375){\line(0,1){190}}
\put(885,375){\line(0,1){190}}
\put(720,400){\makebox(50,100){$\eta$}}
\end{picture}
\end{center}
describes all the faces of $\eta$.
The component of the above diagram at a $2$-category (with Kleisli objects) $\cD$ is
\begin{center} \xext=900 \yext=1050
\begin{picture}(\xext,\yext)(\xoff,\yoff)
\putmorphism(400,950)(0,-1)[\Mnd_{op}(\cD)`\cD`]{950}{0}a
\putmorphism(0,950)(0,-1)[\phantom{\Mnd_{op}(\cD)}`\phantom{\Mnd_{op}(\cD)}`|-|_\cD]{950}{1}l
\putmorphism(800,950)(0,-1)[\phantom{\Mnd_{op}(\cD)}`\phantom{\Mnd_{op}(\cD)}`|-|_\cD]{950}{1}r

\put(620,300){\makebox(50,100){$>$}}
\put(150,340){\line(1,0){500}}
\put(150,360){\line(1,0){500}}
\put(350,200){\makebox(150,100){$U_\cD\circ |F_\cD|$}}

\put(620,600){\makebox(50,100){$>$}}
\put(150,640){\line(1,0){500}}
\put(150,660){\line(1,0){500}}
\put(350,690){\makebox(150,100){$Id_{|-|_\cD}$}}

\put(447,430){\makebox(50,100){$\bigvee$}}
\put(470,445){\line(0,1){150}}
\put(455,475){\line(0,1){120}}
\put(485,475){\line(0,1){120}}
\put(320,450){\makebox(50,100){$\eta_\cD$}}
\end{picture}
\end{center}
The component of the above diagram at a monad $\cS$ in $\cD$ is
\begin{center} \xext=1000 \yext=350
\begin{picture}(\xext,\yext)(\xoff,\yoff)
\putmorphism(0,250)(1,0)[\phantom{C_{[\cD,\cS]}}`\phantom{C_{[\cD,\cS]}}`1_{C_{[\cD,\cS]}}]{1000}{1}a
\putmorphism(0,150)(1,0)[C_{[\cD,\cS]}`C_{[\cD,\cS]}`]{1000}{0}a
\putmorphism(0,50)(1,0)[\phantom{C_{[\cD,\cS]}}`\phantom{C_{[\cD,\cS]}}`\cS]{1000}{1}b

\put(320,100){\makebox(300,100){$\eta_{[\cD,\cS]}\;\;\Da$}}
\end{picture}
\end{center}
This means that $\eta$ is the collection of all the units of all Kleisli adjunctions $F_\cS\dashv U_\cS$ of all the monads
$\cS$ in all the $2$-categories $\cD$ that admit Kleisli objects.

Similarly, $\varepsilon$, defined below, is a perturbation from  $|F|\circ U$ to $Id_\cK$.
\begin{center} \xext=1700 \yext=1050
\begin{picture}(\xext,\yext)(\xoff,\yoff)
\putmorphism(0,950)(1,0)[\phantom{2\bCat_k}`\phantom{2\bCat}`\Mnd_{op}]{1700}{1}a
\putmorphism(0,500)(1,0)[2\bCat_k`2\bCat`]{1700}{0}a
\putmorphism(0,50)(1,0)[\phantom{2\bCat_k}`\phantom{2\bCat}`Emb]{1700}{1}b

\put(400,100){\makebox(50,100){$\Da$}}
\put(415,190){\line(0,1){620}}
\put(435,190){\line(0,1){620}}
\put(300,600){\makebox(50,100){$\cK$}}

\put(1300,100){\makebox(50,100){$\Da$}}
\put(1315,190){\line(0,1){620}}
\put(1335,190){\line(0,1){620}}
\put(1390,600){\makebox(50,100){$\cK$}}

\put(1070,200){\makebox(50,100){$>$}}
\put(600,240){\line(1,0){500}}
\put(600,250){\line(1,0){520}}
\put(600,260){\line(1,0){500}}
\put(800,100){\makebox(150,100){$Id_{\cK}$}}

\put(1070,600){\makebox(50,100){$>$}}
\put(600,640){\line(1,0){500}}
\put(600,650){\line(1,0){520}}
\put(600,660){\line(1,0){500}}
\put(800,690){\makebox(150,100){$|F|\circ U$}}

\put(847,330){\makebox(50,100){$\bigvee$}}
\put(865,345){\line(0,1){220}}
\put(875,345){\line(0,1){220}}
\put(855,375){\line(0,1){190}}
\put(885,375){\line(0,1){190}}
\put(720,400){\makebox(50,100){$\varepsilon$}}
\end{picture}
\end{center}
The component of the above diagram at a $2$-category (with Kleisli objects) $\cD$ is
\begin{center} \xext=900 \yext=1050
\begin{picture}(\xext,\yext)(\xoff,\yoff)
\putmorphism(400,950)(0,-1)[\Mnd_{op}(\cD)`\cD`]{950}{0}a
\putmorphism(0,950)(0,-1)[\phantom{\Mnd_{op}(\cD)}`\phantom{\Mnd_{op}(\cD)}`\cK_\cD]{950}{1}l
\putmorphism(800,950)(0,-1)[\phantom{\Mnd_{op}(\cD)}`\phantom{\Mnd_{op}(\cD)}`\cK_\cD]{950}{1}r

\put(620,300){\makebox(50,100){$>$}}
\put(150,340){\line(1,0){500}}
\put(150,360){\line(1,0){500}}
\put(350,200){\makebox(150,100){$Id_{\cK_\cD}$}}

\put(620,600){\makebox(50,100){$>$}}
\put(150,640){\line(1,0){500}}
\put(150,660){\line(1,0){500}}
\put(350,690){\makebox(150,100){$|F_\cD|\circ U_\cD$}}

\put(447,430){\makebox(50,100){$\bigvee$}}
\put(470,445){\line(0,1){150}}
\put(455,475){\line(0,1){120}}
\put(485,475){\line(0,1){120}}
\put(320,450){\makebox(50,100){$\varepsilon_\cD$}}
\end{picture}
\end{center}

The component of the above diagram at a monad $\cS$ in $\cD$ is
\begin{center} \xext=1800 \yext=350
\begin{picture}(\xext,\yext)(\xoff,\yoff)
\putmorphism(0,250)(1,0)[\phantom{\cK_\cD(\cS)=(C_{[\cD,\cS]})_\cS}`\phantom{(C_{[\cD,\cS]})_\cS=\cK_\cD(\cS)}`F_{[\cD,\cS]}\circ U_{[\cD,\cS]}]{1800}{1}a
\putmorphism(0,150)(1,0)[\cK_\cD(\cS)=(C_{[\cD,\cS]})_\cS`(C_{[\cD,\cS]})_\cS=\cK_\cD(\cS)`]{1800}{0}a
\putmorphism(0,50)(1,0)[\phantom{\cK_\cD(\cS)=(C_{[\cD,\cS]})_\cS}`\phantom{(C_{[\cD,\cS]})_\cS=\cK_\cD(\cS)}`1_{(C_{[\cD,\cS]})_\cS}]{1800}{1}b

\put(720,100){\makebox(300,100){$\varepsilon_{[\cD,\cS]}\;\;\Da$}}
\end{picture}
\end{center}
This means that $\varepsilon$ is the collection of all the counits of all Kleisli adjunctions $F_\cS\dashv U_\cS$ of all the monads $\cS$ in all the $2$-categories $\cD$ that admit Kleisli objects. Needless to say that the perturbations $\eta$ and $\varepsilon$ satisfy the triangular equalities.

\subsection{The categorical description of the Kleisli objects}\label{cat_description}
If $\cD$ is $\bCat$ the $2$-category of categories, then the Kleisli objects coincide with the usual categories of Kleisli algebras.
For a monad $(\cC,\cS,\eta,\mu)$ the category $\cC_\cS$ has the same objects as $\cC$. A morphism in $f:A\ra B$ in $\cC_\cS$ is a morphism
in $f:A\ra \cS(A)$ with the usual identities, compositions, $U_\cS$, and $F_\cS$. The component at $X$ in $\cC$ of the
natural transformation $\kappa : F_\cS\circ \cS\ra F_\cS$ is $1_{\cS(X)}$.

If a $2$-category $\cD$ admits Kleisli objects we can ask whether the Kleisli $2$-functor $\cK :\Mnd_{op}(\cD)\ra \cD$ preserves limits
of a particular kind. We have

\begin{lemma}
 The Kleisli $2$-functor $\cK :\Mnd_{op}(\bCat)\ra \bCat$ preserves products of $0$-cells.
\end{lemma}

\begin{proof} We will sketch the construction for binary products. Let $(\cC,\cS,\eta,\mu)$ and $(\cC',\cS',\eta',\mu')$ in $\bCat$.
Then their product in $\Mnd_{op}(\bCat)$ is $(\cC\times \cC',\cS\times\cS',(\eta,\eta'),(\mu,\mu'))$. One can easily verify that the unique morphism
$$H: (\cC\times \cC')_{\cS\times\cS'} \lra \cC_\cS\times \cC'_{\cS'}$$
such that $H\circ F_{\cS\times\cS'}=F_\cS\times F_{\cS'}$ and $H(\kappa^{\cS\times \cS'})= (\kappa^{\cS},\kappa^{\cS'})$ is an isomorphism.
\end{proof}  
\vskip 2mm
{\em Remark.} Note that, as the $2$-functor $EM :\Mnd(\bCat)\ra \bCat$ is a right $2$-adjoint it preserves all limits.

\subsection{The standard Kleisli and EM objects}
Suppose that we have a $2$-functor $G:\cD\ra \cE$ between two $2$-categories that admit Kleisli objects. Thus, we can form a diagram
\begin{center} \xext=1000 \yext=950
\begin{picture}(\xext,\yext)(\xoff,\yoff)
\setsqparms[-1`1`1`-1;1000`600]
 \putsquare(0,150)[\Mnd_{op}(\cD)`\cD`\Mnd_{op}(\cE)`\cE;\iota_{op}`\Mnd_{op}(G)`G`\iota_{op}]

 \putmorphism(0,50)(1,0)[\phantom{\Mnd_{op}(\cE)}`\phantom{\cE}`{|-|}]{1000}{1}b
 \putmorphism(0,250)(1,0)[\phantom{\Mnd_{op}(\cE)}`\phantom{\cE}`{\cal K}]{1000}{1}a

 \putmorphism(0,680)(1,0)[\phantom{\Mnd_{op}(\cD)}`\phantom{\cD}`{|-|}]{1000}{1}b
 \putmorphism(0,880)(1,0)[\phantom{\Mnd_{op}(\cD)}`\phantom{\cD}`{\cal K}]{1000}{1}a
 \end{picture}
\end{center}
so that the squares with $\iota_{op}$'s and $|-|$'s that commute. If it happen that the square with $\cK$'s
commute up to the canonical isomorphism, we say that $\cD$ {\em has standard Kleisli objects with respect to} $G$,
c.f. \cite{McC}. If $G$ is understood then we say that $\cD$ {\em has standard Kleisli objects with respect to} $\cE$.
The standard Kleisli objects with respect to $\bCat$ (and an obvious forgetful functor) will be called  {\em standard Kleisli objects}.
The standard EM objects are defined in a similar way.
\section{Monoidal objects in $2$-categories}

Let $\cD$ be a $2$-category with finite products of $0$-cells. In such a $2$-category $\cD$, we can talk about monoidal objects, (op)lax monoidal $1$-cells, and monoidal $2$-cells, as we talk about monoidal categories, (op)lax monoidal functors, and monoidal natural transformations in the $2$-category $\bCat$.
A {\em monoidal object in} $\cD$ consists of a $0$-cell $\cC$, two $1$-cells $\otimes : \cC\times \cC\lra \cC$, $I:1\ra \cC$, and three invertible $2$-cells
\[ \alpha: \otimes\circ (1\times \otimes) \Ra \otimes\circ (\otimes \times 1)\hskip 1cm \lambda: \otimes \circ \lk I, 1_\cC\rk \Ra 1_\cC\hskip 1cm \rho: \otimes \circ \lk 1_\cC , I\rk \Ra 1_\cC\]
making the pentagon
\[ \otimes \lk \alpha,1\rk \circ \alpha_{1_\cC\times\otimes\times 1_\cC}\circ \otimes\lk 1,\alpha \rk
= \alpha_{\otimes \times 1_\cC\times 1_\cC}\circ \alpha_{1_\cC\times 1_\cC\times\otimes} \]
and the triangle
\[ \otimes \lk \varrho_{\pi_1},1_{\pi_2} \rk \circ \alpha_{\lk \pi_1,I,\pi_2 \rk} = \otimes \lk 1_{\pi_1},\lambda_{\pi_2} \rk \]
commute, where $\lk \pi_1,I,\pi_2 \rk : \cC\times \cC\lra \cC\times \cC\times \cC$ is the obvious morphism.

A {\em lax monoidal morphism of monoidal objects}
$$(F,\varphi,\bar{\varphi}): (\cC,\otimes,I,\alpha,\lambda,\varrho) \lra (\cC',\otimes',I',\alpha',\lambda',\varrho')$$
consists of a $1$-cell and two $2$-cells
$$F: \cC\ra \cC',  \hskip 1cm \bar{\varphi}:I' \Ra F\circ I,  \hskip 1cm \varphi: \otimes'\circ (F\times F)\Ra F \circ\otimes$$
such that the following three diagrams
\begin{center} \xext=2300 \yext=1400
\begin{picture}(\xext,\yext)(\xoff,\yoff)
\setsqparms[1`1`1`0;2000`600]
 \putsquare(0,650)[\otimes'\circ (1\times \otimes')\circ (F\times F\times F)`
 \otimes'\circ (\otimes'\times 1)\circ (F\times F\times F)`
  \otimes'\circ (F\times F)\circ(1\times \otimes)`
 \otimes'\circ (F\times F)\circ(\otimes\times 1);
 \alpha'_{F\times F\times F}`
 \otimes'( 1,\varphi)`
 \otimes'( \varphi,1)`]

 \setsqparms[0`1`1`1;2000`600]
 \putsquare(0,50)[\phantom{ \otimes'\circ (F\times F)\circ(1\times \otimes)}`
 \phantom{\otimes'\circ (F\times F)\circ(\otimes\times 1)}`
 F\circ \otimes\circ(1\times \otimes)`
  F\circ \otimes\circ(\otimes\times 1);` \varphi_{(1\times \otimes)}`
 \varphi_{(\otimes\times 1)}`
F(\alpha)]
 \end{picture}
\end{center}
\begin{center} \xext=1500 \yext=750
\begin{picture}(\xext,\yext)(\xoff,\yoff)
\setsqparms[1`1`-1`1;1500`600]
 \putsquare(0,50)[\otimes'\circ \lk 1_\cC,I'\rk\circ F`F`
  \otimes'\circ (F\times F)\circ\lk 1_\cC,I\rk`
 F\circ\otimes\circ \lk 1_\cC,I\rk;
 \rho'_F`
 \otimes'( 1,\bar{\varphi})`
 F(\rho)`
 \varphi_{\lk 1_\cC,I\rk}]
 \end{picture}
\end{center}
and
\begin{center} \xext=1500 \yext=750
\begin{picture}(\xext,\yext)(\xoff,\yoff)
\setsqparms[1`1`-1`1;1500`600]
 \putsquare(0,50)[\otimes'\circ \lk I',1_\cC\rk\circ F`F`
  \otimes'\circ (F\times F)\circ\lk I, 1_\cC\rk`
 F\circ\otimes\circ \lk I,1_\cC\rk;
 \lambda'_F`
 \otimes'( \bar{\varphi},1)`
 F(\lambda)`
 \varphi_{\lk I,1_\cC\rk}]
 \end{picture}
\end{center}
commute.

An {\em oplax monoidal morphism of monoidal objects}
$$(F,\varphi,\bar{\varphi}): (\cC,\otimes,I,\alpha,\lambda,\varrho) \lra (\cC',\otimes',I',\alpha',\lambda',\varrho')$$
consists of a $1$-cell
and two $2$-cells
$$F: \cC\ra \cC', \hskip 1cm\bar{\varphi}:F\circ I \Ra I', \hskip 1cm \varphi:  F \circ\otimes \Ra \otimes'\circ (F\times F)$$
(note the change of direction!) satisfying similar diagrams as those for lax monoidal morphism.

A {\em transformation of lax monoidal morphism}
$$\tau : (F,\varphi,\bar{\varphi})\Ra (F',\varphi',\bar{\varphi}')$$
is a $2$-cell $\tau : F\ra F'$ such that the diagrams
\begin{center} \xext=2700 \yext=650
\begin{picture}(\xext,\yext)(\xoff,\yoff)
\setsqparms[1`1`-1`1;1200`500]
 \putsquare(0,50)[\otimes'\circ (F\times F)`\otimes'\circ (F'\times F')`
  F\circ \otimes` F'\circ \otimes;
 \otimes'(\sigma,\sigma)`
 \varphi`
 \varphi'`
 \sigma_\otimes]

 \putmorphism(2490,580)(0,-1)[F\circ I`F'\circ I`\sigma_I]{460}{1}r
\put(1850,400){\vector(3,1){500}}
\put(1850,300){\vector(3,-1){500}}
\put(1700,300){\makebox(150,100){$I'$}}

\put(1950,500){\makebox(150,100){$\bar{\varphi}$}}
\put(1950,100){\makebox(150,100){$\bar{\varphi}'$}}
 \end{picture}
\end{center}
commute. The transformations of oplax monoidal morphism are defined similarly.

Recall from \ref{some3cat} that $2\bCat_\times$ is the $3$-category of $2$-categories with finite products. We have $3$-functors
\[ \Mon,\Mon_{op} : 2\bCat_\times \lra 2\bCat_\times \]
$\Mon$ ($\Mon_{op}$) sends a $2$-category $\cD$ with finite products to the $2$-category $\Mon(\cD)$ ($\Mon_{op}(\cD)$)  of monoidal objects,
(op)lax  monoidal morphism, and their transformations.  We also have $3$-transformations
\[ \cU:\Mon\Ra \Id, \hskip 2cm \cU_{op}:\Mon_{op}\Ra \Id \]
whose components are forgetful functors forgetting the monoidal structure. $\Id$ is the identity functor on $2\bCat_\times$.

The following theorem says that, in any $2$-category $\cD$ with finite products, monoidal monads  are 'the same things' as monoidal categories
in the $2$-category of monads over $\cD$. However there are subtleties concerning (op)laxness of  $1$-cells.

\begin{lemma}\label{comm_structures}
The following diagrams of $3$-functors
\begin{center} \xext=2600 \yext=500
\begin{picture}(\xext,\yext)(\xoff,\yoff)
\setsqparms[1`1`1`1;800`400]
 \putsquare(0,0)[2\bCat_\times `2\bCat_\times `2\bCat_\times `2\bCat_\times ;\Mon`\Mnd_{op}`\Mnd_{op}`\Mon]
  \putsquare(1800,0)[2\bCat_\times `2\bCat_\times `2\bCat_\times `2\bCat_\times ;\Mon_{op}`\Mnd`\Mnd`\Mon_{op}]
 \end{picture}
\end{center}
commute up to natural $3$-isomorphisms $\xi$ and $\xi'$, respectively. Moreover, these isomorphisms are compatible with $3$-transformation
$\iota$ and $\cU$ in the sense that the diagrams of $3$-transformations
\begin{center} \xext=2200 \yext700
\begin{picture}(\xext,\yext)(\xoff,\yoff)
\putmorphism(1100,650)(0,-1)[\Mnd_{op}\Mon`\Mon\Mnd_{op}`\xi]{600}{1}r
\put(200,400){\vector(3,1){600}}
\put(200,300){\vector(3,-1){600}}
\put(0,300){\makebox(150,100){$\Mon$}}

\put(220,530){\makebox(150,100){$(\iota_{op})_\Mon$}}
\put(220,70){\makebox(150,100){$\Mon(\iota_{op})$}}

\put(1380,600){\vector(3,-1){600}}
\put(1380,100){\vector(3,1){600}}
\put(2100,300){\makebox(150,100){$\Mnd_{op}$}}

\put(1750,520){\makebox(150,100){$\Mnd_{op}(\cU)$}}
\put(1700,50){\makebox(150,100){$\cU_{\Mnd_{op}}$}}
\end{picture}
\end{center}

\begin{center} \xext=2200 \yext750
\begin{picture}(\xext,\yext)(\xoff,\yoff)
\putmorphism(1100,650)(0,-1)[\Mnd\Mon_{op}`\Mon_{op}\Mnd`\xi']{600}{1}r
\put(250,400){\vector(3,1){600}}
\put(250,300){\vector(3,-1){600}}
\put(0,300){\makebox(150,100){$\Mon_{op}$}}

\put(270,530){\makebox(150,100){$(\iota)_{\Mon_{op}}$}}
\put(270,70){\makebox(150,100){$\Mon_{op}(\iota)$}}

\put(1430,600){\vector(3,-1){600}}
\put(1430,100){\vector(3,1){600}}
\put(2100,300){\makebox(150,100){$\Mnd$}}

\put(1800,520){\makebox(150,100){$\Mnd(\cU_{op})$}}
\put(1800,50){\makebox(150,100){$(\cU_{op})_\Mnd$}}

\end{picture}
\end{center}
commute.
\end{lemma}

 \begin{proof}  For any $2$-category $\cD$ with products the cells in the $2$-categories $\Mon\Mnd_{op}(\cD)$, $\Mnd_{op}\Mon(\cD)$,
$\Mon_{op}\Mnd(\cD)$,  $\Mnd\Mon_{op}(\cD)$ are tuples of cells from $\cD$ satisfying certain (equational) coherence conditions.
An easy but long verification shows, for example, that $0$-cells of both  $\Mon\Mnd_{op}(\cD)$, $\Mnd_{op}\Mon(\cD)$ are tuples of cells
that differ only by the cells order, but not the conditions they satisfy. Similarly for $1$- and $2$-cells. The morphism
$\xi'$ just permutes these tuples. One can easily check that this 'permutation isomorphism' is compatible with both $\iota$ and $\cU$,
as stated in the theorem.

More explicitly, we can identify $0$-cells of both $\Mon\Mnd_{op}(\cD)$ and $\Mnd_{op}\Mon(\cD)$ as $11$-tuples
$$(\cC,\otimes,I,\alpha,\lambda,\varrho,\cS,\varphi,\bar{\varphi},\eta,\mu)$$ satisfying
certain conditions that we explain below.

In both cases $(\cC,\cS,\eta,\mu)$ must be a monad. Moreover, in  $\Mon\Mnd_{op}(\cD)$
\[ (\otimes,\varphi): (\cC\times \cC,\cS\times \cS,(\eta,\eta),(\mu,\mu)) \lra (\cC,\cS,\eta,\mu) \]
\[  (I,\bar{\varphi}):(\cC,1,1,1)\lra (\cC,\cS,\eta,\mu) \]
must be oplax morphisms of monads. This condition is equivalent to the condition that
\[\eta:(1_\cC,1,1)\lra (\cS,\varphi,\bar{\varphi}),\hskip 5mm
\mu :(\cS^2,\cS(\varphi)\circ \varphi_{\cS\times \cS},\cS(\bar{\varphi})\circ\bar{\varphi})\lra (\cS,\varphi,\bar{\varphi}) \]
are monoidal transformations of lax monoidal morphisms. The later condition is required for such tuple to be in $\Mnd_{op}\Mon(\cD)$.
Finally, the conditions that
\[ \alpha : (\otimes\circ (1\times \otimes),\varphi_{1\times\otimes}\circ\otimes(1,\varphi))\lra
           (\otimes\circ (\otimes\times 1),\varphi_{\otimes\times 1}\circ\otimes(\varphi,1)) \]
\[ \lambda :  (\otimes\circ\lk I,1_\cC \rk , \varphi_{\lk I,1_\cC\rk}\circ \otimes(\bar{\varphi},1_\cS))\lra (1_\cC,1_\cS) \]
\[ \varrho :  (\otimes\circ\lk 1_\cC, I\rk , \varphi_{\lk 1_\cC,I\rk}\circ \otimes(1_\cS,\bar{\varphi}))\lra (1_\cC,1_\cS)\]
are transformations of oplax morphisms of monads, required for the tuple to be
in $\Mon\Mnd_{op}(\cD)$ is equivalent to the condition that
$$(\cS,\varphi,\bar{\varphi}): (\cC,\otimes,I,\alpha,\lambda,\varrho) \lra (\cC,\otimes,I,\alpha,\lambda,\varrho)$$
is a lax monoidal morphisms. This is another condition required for the tuple to be in  $\Mnd_{op}\Mon(\cD)$.

In that sense the conditions imposed on such $11$-tuple to be either in $\Mon\Mnd_{op}(\cD)$ or $\Mnd_{op}\Mon(\cD)$ are the same.
The similar thing happen with $1-$ and $2$-cells in those $2$-categories. Thus they are isomorphic.

The remaining details are left for the reader. \end{proof} 

{\em Remark.} This fact is a fragment of a much wider phenomena, deserving a serious independent studies,
that if we combine together two 'algebraic structures' then they cooperate well when one is taken with
lax morphisms and the other with oplax morphisms like $\Mnd_{op}\Mnd\cong\Mnd\Mnd_{op}$, $\Mon_{op}\Mon\cong\Mon\Mon_{op}$.

\section{The Kleisli objects in $2$-categories of monoidal objects}
In this section we give a $3$-categorical proof of

\begin{theorem}\label{Kleisli_moncat} Let $\cD$ be a $2$-category with finite products that admits Kleisli objects.
 Then the $2$-category $\Mon(\cD)$ admits Kleisli object and they are standard with respect to $\cD$.
\end{theorem}
 \begin{proof}
 Using Lemma \ref{comm_structures} and the fact that $\cU$ is $3$-natural transformation we see that that
the diagram of $3$-functors and $3$-natural transformation
\begin{center} \xext=2300 \yext=1300
\begin{picture}(\xext,\yext)(\xoff,\yoff)
\setsqparms[-1`1`1`1;1200`600]
 \putsquare(0,50)[\Mon`\Mon\Mnd_{op}`\Id`\Mnd_{op};\Mon(\cK)`\cU`\cU_{\Mnd_{op}}`\iota_{op}]
 \putmorphism(0,100)(1,0)[\phantom{\Id}`\phantom{\Mnd_{op}}`\cK]{1200}{-1}a
 \putmorphism(0,600)(1,0)[\phantom{\Mon}`\phantom{\Mon\Mnd_{op}}`\Mon(\iota_{op})]{1200}{1}b
  \put(1700,1180){\makebox(300,100){$\Mnd_{op}\Mon$}}
  \put(1260,960){\makebox(300,100){$\xi$}}
  \put(1980,680){\makebox(300,100){$\Mnd_{op}(\cU)$}}
  \put(380,1040){\makebox(300,100){$(\iota_{op})_\Mon$}}

\put(1650,1150){\vector(-1,-1){400}}
\put(1850,350){\vector(-2,-1){500}}
\put(1850,350){\line(0,1){790}}

  \put(0,730){\line(3,2){300}}
  \put(300,930){\vector(4,1){1200}}
 \end{picture}
\end{center}
commutes. Evaluating this diagram at a $2$-category  $\cD$  with finite products that admits Kleisli objects we get a commuting
diagram of $2$-categories and $2$-functors
\begin{center} \xext=2300 \yext=1300
\begin{picture}(\xext,\yext)(\xoff,\yoff)
\setsqparms[-1`1`1`1;1200`600]
 \putsquare(0,50)[\Mon(\cD)`\Mon\Mnd_{op}(\cD)`\cD`\Mnd_{op}(\cD);\Mon(\cK_\cD)`\cU_\cD`\cU_{\Mnd_{op}(\cD)}`\iota_{op,\cD}]
 \putmorphism(0,100)(1,0)[\phantom{\cD}`\phantom{\Mnd_{op}(\cD)}`\cK_\cD]{1200}{-1}a
 \putmorphism(0,600)(1,0)[\phantom{\Mon(\cD)}`\phantom{\Mon\Mnd_{op}(\cD)}`\Mon(\iota_{op,\cD})]{1200}{1}b
  \put(1700,1180){\makebox(300,100){$\Mnd_{op}\Mon(\cD)$}}
  \put(1260,960){\makebox(300,100){$\xi_\cD$}}
  \put(2000,680){\makebox(300,100){$\Mnd_{op}(\cU_\cD)$}}
  \put(380,1070){\makebox(300,100){$(\iota_{op})_{\Mon(\cD)}$}}

\put(1650,1150){\vector(-1,-1){400}}
\put(1850,350){\vector(-2,-1){470}}
\put(1850,350){\line(0,1){790}}

  \put(0,730){\line(3,2){300}}
  \put(300,930){\vector(4,1){1150}}
 \end{picture}
\end{center}
As $\cK_{\cD}\dashv \iota_{op,\cD}$ and the $3$-functor $\Mon$ preserves $2$-adjunctions, we have $\Mon(\cK_{\cD})\dashv \Mon(\iota_{op,\cD})$.
Since $\xi_\cD$ is an isomorphism we get that $\Mon(\cK_{\cD})\circ\xi_\cD\dashv \iota_{op,\Mon(\cD)}$, i.e. $\Mon(\cD)$ indeed admits Kleisli objects.
 \end{proof}
 
As $\bCat$ has products of $0$-cells and admits Kleisli objects we get
\begin{corollary}
 The $2$-category $\Mon(\bCat)$ admits standard Kleisli objects.
\end{corollary}

Note that Theorem \ref{Kleisli_moncat} can be rephrased in a slightly more general form as a lifting property.

\begin{theorem}\label{Kleisli_moncat1} The $3$-functor $\Mon$ and the $3$-transformation $\cU$ can be lifted to
$2\bCat_{\times k}$ so that the diagram of $3$-categories, $3$-functors, and $3$-transformations commutes up to a canonical isomorphism
\begin{center} \xext=1800 \yext=1150
\begin{picture}(\xext,\yext)(\xoff,\yoff)

\putmorphism(0,1050)(1,0)[\phantom{2\bCat_{\times k}}`\phantom{2\bCat_{\times k}}`\overline{\Mon}]{1800}{1}a
\putmorphism(0,850)(1,0)[\phantom{2\bCat_{\times k}}`\phantom{2\bCat_{\times k}}`Id]{1800}{1}b
\putmorphism(0,950)(1,0)[2\bCat_{\times k}`2\bCat_{\times k}`]{1800}{0}b

\put(950,900){\makebox(300,100){$\overline{\cU}\;\Da$}}

\putmorphism(0,250)(1,0)[\phantom{2\bCat_{\times}}`\phantom{2\bCat_{\times}}`\Mon]{1800}{1}a
\putmorphism(0,50)(1,0)[\phantom{2\bCat_{\times}}`\phantom{2\bCat_{\times}}`Id]{1800}{1}b
\putmorphism(0,150)(1,0)[2\bCat_{\times}`2\bCat_{\times}`]{1800}{0}b

\put(950,100){\makebox(300,100){$\cU\;\Da$}}

\put(0,850){\vector(0,-1){600}}
\put(1750,850){\vector(0,-1){600}}
\end{picture}
\end{center}
\end{theorem}
 \begin{proof}
 From Theorem \ref{Kleisli_moncat}, we know that, if we apply the $3$-functor $\Mon$ to
a $2$-category $\cD$ that has not only finite products but also Kleisli objects, then we will get
$\Mon(\cD)$ that also has finite products and Kleisli objects. We need also to verify that $\Mon$ applied
to a $2$-functor $F:\cD\ra \cD'$ that preserve Kleisli objects also preserves Kleisli objects. This can be proved
using a similar argument as the one in Theorem \ref{Kleisli_moncat}. We leave it for the reader.  \end{proof}

\section{The EM objects in $2$-categories of monoidal objects}
The dual statement of Theorem \ref{Kleisli_moncat} is

\begin{theorem}\label{EM_moncat} Let $\cD$ be a $2$-category with finite products admitting EM objects.
 Then the $2$-category $\Mon_{op}(\cD)$ admits EM objects and they are standard with respect to $\cD$.
\end{theorem}

Putting $\cD$ to be $\bCat$ in the above Theorem, we obtain a result by I. Moerdijk \cite{Mo} in a sharper version
of P. McCrudden \cite{McC}

\begin{corollary} [Moerdijk, McCrudden]
 The $2$-category $\Mon_{op}(\bCat)$ admits standard EM objects.
\end{corollary}

\section{Some other algebraic structures}

If we replace the $3$-functor $\Mon$ ($\Mon_{op}$) by the $3$-functor  $\BMon$ ($\BMon_{op}$) of braided monoidal
objects with lax (oplax) monoidal morphisms and monoidal transformations or $3$-functor $\SMon$ ($\SMon_{op}$)
of symmetric monoidal objects with lax (oplax) monoidal morphisms and monoidal transformations or $3$-functor $\Cmd$ ($\Cmd_{op})$
of comonads with lax (oplax) monoidal morphisms and transformations, or even $3$-functor $\Mnd$ ($\Mnd_{op})$,
we can repeat the whole reasoning again. In this way we obtain

\begin{theorem}\label{Kleisli_moncat2}

Let $\cD$ be a $2$-category that admits Kleisli objects.
Then the $2$-categories $\Mnd(\cD)$ and $\Cmd(\cD)$ admit Kleisli object and they are standard with respect to $\cD$.

Moreover, if $\cD$ has finite products, then  the $2$-categories $\BMon(\cD)$, $\SMon(\cD)$
admit Kleisli object and they are standard with respect to $\cD$.
\end{theorem}

\begin{theorem}\label{EM_moncat1} Let $\cD$ be a $2$-category that admits EM objects.
Then the $2$-category $\Mnd_{op}(\cD)$ and $\Cmd_{op}(\cD)$ admit EM objects and they are standard with respect to $\cD$.

Moreover, if $\cD$ has finite products, then  the $2$-categories $\BMon_{op}(\cD)$, $\SMon_{op}(\cD)$
admit EM objects and they are standard with respect to $\cD$.
\end{theorem}

{\em Remarks.}
\begin{enumerate}
  \item The above facts suggest that the results of this paper can be still generalized.
One way is to axiomatize the formal properties of the relation of $3$-functors
$\Mon$, $\BMon$, $\SMon$, $\Mnd(\cD)$, and $\Cmd(\cD)$  with respect to the $3$-functor $\Mnd_{op}$
and the relation of $3$-functors  $\Mon_{op}$, $\BMon_{op}$, $\SMon_{op}$, $\Mnd(\cD)$, and $\Cmd(\cD)$  with respect to the $3$-functor
$\Mnd$ and get this way still more abstract statement.  This would  be worth trying if there were found some new natural examples,
other than iterations of the $3$-functors listed above.
  \item The other more specific generalization would be to show that `any' algebraic $2$-categorical structure will do.
The precise formulation what such algebraic structure should be is still to be found.
The work of M. Hyland and his coworkers \cite{Hy} might be of a help.

\end{enumerate}


\begin{thebibliography}{HMP}

\bibitem[Day]{Day}
B. Day, {\em On Closed category of Functors II}. Proceedings Sydney Category Theory Seminar 1972/1973,
Lecture Notes in Math no 420, Springer-Verlag, 1974, pp. 20–54.

\bibitem[Gray]{Gray}
J. Gray, {\em Formal Category Theory: Adjointness for $2$-Categorires}. Lecture Notes in Math no 391, Springer-Verlag, 1974.

\bibitem[Hy]{Hy}
M. Hyland, {\em Sorts of algebraic theories}. Talk at International Workshop on Categorical Logic
August 28–29, 2010. Masaryk University, Brno, Czech Republic.

\bibitem[McC]{McC}
P. McCrudden, {\em Opmonoidal monads}. Theory and Applications of Categories, Vol. 10, No. 19, 2002, pp. 469–485.

\bibitem[Mo]{Mo}
I. Moerdijk, {\em Monads on tensor categories}. Journal of Pure and Applied Algebra 168 (2002) 189-208.

\bibitem[St]{St}
R. Street, {\em The Formal Theory of Monads}. Journal of Pure and Applied Algebra 2 (1972) 149-168.

\bibitem[Z]{Z}
M. Zawadowski, {\em Lax Monoidal Fibrations}, 	arXiv:0912.4464v2 [math.CT], accepted for the M.Makkai 70th birthday volume.

\end{thebibliography}
\end{document}